\documentclass{article}
\usepackage[a4paper, verbose]{geometry}

\usepackage{layout} %

\usepackage[affil-it]{authblk}

\usepackage[utf8]{inputenc}
\usepackage[T1]{fontenc}

\usepackage{enumitem}

\usepackage{float} %

\usepackage[margin=16pt, labelfont=bf]{caption} %
\usepackage{subcaption}

\usepackage{graphicx} 

\usepackage{booktabs}
\usepackage{siunitx}

\usepackage{amsmath}
\usepackage{breqn}
\usepackage{amssymb} %
\usepackage{amsthm} %
\usepackage{mathtools} %
\usepackage{bbm} %

\newtheorem{theorem}{Theorem}[section]
\newtheorem{corollary}[theorem]{Corollary}
\newtheorem{lemma}[theorem]{Lemma}
\newtheorem{remark}[theorem]{Remark}
\newtheorem{conjecture}[theorem]{Conjecture}
\newtheorem{definition}[theorem]{Definition}

\DeclareMathOperator*{\argmin}{arg\,min} %
\DeclareMathOperator*{\argmax}{arg\,max} %

{
    }
{ 
    }

\newcommand{\overbar}[1]{\overline{#1}}

\DeclarePairedDelimiterXPP\expect[1]{\operatorname{\mathbb{E}}}{[}{]}{}{\ifblank{#1}{\:\cdot\:}{#1}}
\DeclarePairedDelimiterXPP\variance[1]{\operatorname{Var}}{(}{)}{}{\ifblank{#1}{\:\cdot\:}{#1}}
\DeclarePairedDelimiter\abs{\lvert}{\rvert}
\DeclarePairedDelimiterXPP\covariance[1]{\operatorname{Cov}}{(}{)}{}{\ifblank{#1}{\:\cdot\:}{#1}}
\DeclarePairedDelimiterXPP\average[2]{}{\langle}{\rangle}{\ifblank{#2}{}{_{#2}}}{\ifblank{#1}{\:\cdot\:}{#1}}
\DeclarePairedDelimiterXPP\proba[1]{\operatorname{P}}{(}{)}{}{\ifblank{#1}{\:\cdot\:}{#1}}

\usepackage{parskip}

\usepackage{hyperref}
\hypersetup{ %
    colorlinks=true,
    linkcolor=blue,
    filecolor=magenta,      
    urlcolor=cyan,
    pdfpagemode=FullScreen,
    }

\title{On some features of Quadratic Unconstrained Binary Optimization with random coefficients}

\author[1]{Marco Isopi} 
\author[2]{Benedetto Scoppola}
\author[3]{Alessio Troiani}

\affil[1]{Dipartimento di Matematica, Sapienza Università di Roma, Piazzale Aldo Moro 5, 00185 Roma, Italy}
\affil[2]{Dipartimento di Matematica, Universita di Roma ``Tor Vergata'', Via della Ricerca Scientifica 1, 00133 Roma, Italy}
\affil[3]{Dipartimento di Matematica e Informatica, Università degli Studi di Perugia, via Vanvitelli, 1 06123 Perugia, Italy}

\date{\today}

\usepackage{todonotes} %

\begin{document}

\maketitle

\begin{abstract}
    Quadratic Unconstrained Binary Optimization (QUBO or UBQP) is concerned with 
    maximizing/minimizing the quadratic form
    $H(J, \eta) = W \sum_{i,j}  J_{i,j} \eta_{i} \eta_{j}$ with $J$ a matrix of coefficients,
    $\eta \in \{0, 1\}^N$ and $W$ a normalizing constant.
    In the statistical mechanics literature, QUBO is a lattice gas counterpart to the (generalized) Sherrington--Kirkpatrick spin glass model.
    Finding the optima of $H$ is an NP-hard problem. 
    Several problems in combinatorial optimization and data analysis can be mapped to QUBO in a straightforward manner.    
    In the combinatorial optimization literature, random instances of QUBO are often used to test the effectiveness of heuristic algorithms.
    
    Here we consider QUBO with random independent coefficients and show that if
    the $J_{i,j}$'s have zero mean and finite variance then, after proper normalization, 
    the minimum and
    maximum \emph{per particle} of $H$ do not depend on the details of the distribution
    of the couplings and are concentrated around their expected values.
    Further, with the help of numerical simulations, we study the 
    minimum and maximum of the objective function and provide some insight
    into the structure of the minimizer and the maximizer of $H$. 
    We argue that also this structure is rather robust.
    Our findings hold also in the diluted case where each of the $J_{i,j}$'s is allowed
    to be zero with probability going to $1$ as $N \to \infty$ in a suitable way.

    \vspace{1em}
 \end{abstract}

\section{Introduction}\label{sec:introduction}

We consider the quadratic form
\begin{equation}
\label{eq:general_quadratic_form}
	H(J, \eta) = W \sum_{\mathclap{\substack{i,j\\ 1 \leq i,j \leq N}}}  J_{i,j} \eta_{i} \eta_{j}
\end{equation}
where $J \in \mathbb{R}^{N \times N}$ is a matrix of coefficients, $\eta \in
\mathcal{A}^N$ (with $\mathcal{A}$ a finite subset of $\mathbb{R}$), $\eta_{i}$ is the value of
the $i$-th component of $\eta$, and $W \in \mathbb{R}$ is a suitable
normalizing constant. This quadratic form plays an important role both in
combinatorial optimization and statistical mechanics.

In statistical mechanics \eqref{eq:general_quadratic_form}
is the Hamiltonian of several physical systems whose nature
depends on the elements of $\mathcal{A}$.
For instance, if $\mathcal{A} = \{-1, 1\}$ the Hamiltonian 
describes a system of (pairwise) interacting spins whereas
if $\mathcal{A} = \{0, 1\}$ it is generally used to describe
a system of (pairwise) interacting particles.
Spins or particles live on the vertices of some,
possibly oriented, graph $\mathcal{G} = \{V, E\}$, 
called the interaction graph,
with $|V| = N$ and $J$ the weighted adjacency matrix
of $\mathcal{G}$. For each $(i,j) \in E$ the entry $J_{i,j}$
represents the strength of the interaction
between the entities (spins or particles) at vertices $i$
and $j$ of $\mathcal{G}$. The microscopic state of the physical system is given by $\eta$. 

The physical system is described in terms 
of the probability of its microscopic states (Gibbs measure): 
\begin{equation}\label{eq:gibbs_measure}
    \mu(\eta) = \frac{e^{-\beta H(J, \eta)}}{Z_{\beta}}
\end{equation}
where the parameter $\beta$ is called the \emph{inverse temperature} and
$Z_{\beta}$ is a normalizing constant called partition function.

In combinatorial optimization the problem of minimizing (or
maximizing) \eqref{eq:general_quadratic_form} when
$\mathcal{A} = \{0, 1\}$ and $W = 1$ is known under the
names Quadratic Unconstrained Binary Optimization (QUBO) or
Unconstrained Binary Quadratic Programming (UBQP). 
QUBO is NP-hard and, in general, no polynomial time algorithm is known
to find a minimizer of $H$.
Many problems in combinatorial optimization and
data analysis can be mapped to QUBO in a straightforward
manner (see \cite{glover2019quantum} for a survey). 
Even constrained optimization problems can be reduced to QUBO
by introducing quadratic penalties in the objective
function.

Minimizers (\emph{ground states}) of forms like $H$
in \eqref{eq:general_quadratic_form}
are of relevance in the context of
statistical mechanics as well. Indeed 
ground states are the ones with the highest
probability with respect to the probability measure
\eqref{eq:gibbs_measure}. As the temperature of the physical
system approaches zero ($\beta \to \infty$), the system
will eventually reach its ground state.

If the signs of the entries of $J$ are \emph{disordered}, i.e. if it is not possible to find the global minimizer with a local criterion, finding
the ground states of the system is non-trivial.

In the context of statistical mechanics, there is a vast
literature concerning the properties of the ground states
of $H$ of the form \eqref{eq:general_quadratic_form}. Here we recall the comprehensive texts
by Parisi, Mezard and Virasoro \cite{mezard1987spin} (for a more physical perspective) and more recently by Talagrand \cite{talagrand2010mean, talagrand2010mean2} and Panchenko \cite{panchenko2012sherrington}. The recent book \cite{charbonneau2023spin} provides an up to date overview. 

The system that attracted more effort from physicists, and that is most described in the references above, is the so called
Sherrington--Kirkpatrick (SK) model, in which $\eta \in \{-1, 1\}^{N}$,
$W = \frac{1}{\sqrt{N}}$ and $J_{i,j}$ are independent
standard Gaussian random variables. In relatively more recent times, always in the context of spin systems, the universality of the features of the SK model with respect to the distribution of the (independent) $J_{ij}$ has been proven, see \cite{carmona2006universality, chatterjee2005simple}.

The $\{0, 1\}$ counterpart to the Sherrington--Kirkpatrick
model, whose Hamiltonian matches the objective function of a QUBO instance, has not been the subject of the same vast attention in the statistical mechanics literature.
Many important results have been achieved in \cite{panchenko2005generalized, panchenko2018mixed}, in which features of the minimizers and maximizers analogous to the SK model for the Hamiltonian \eqref{eq:general_quadratic_form} have been proven for $J_{i,j}$ independent Gaussian random variables and a quite general choice of the set $\mathcal{A}$, including also the particle case $\eta \in \{0, 1\}^{N}$. In particular it has been proven that the free energy of the system is close to its expectation with probability one (concentration) and that 
its thermodynamic limit exists.

The same problem, in the canonical ensemble, i.e. with a fixed number of particles, has been discussed in \cite{erba2024statistical}. There, for general distribution of $J_{i,j}$, interesting features of the system, most of all for small density of particles, has been found.

Restricting only to the case $\eta \in \{0, 1\}^N$ and $J_{i,j}$ independent Gaussian random variables, with very different and more elementary techniques, in \cite{STgmf} an almost sure lower bound for the minimum per particle of $H$ has been provided together with a different proof of concentration. 

In this paper we will show, in the spirit of \cite{carmona2006universality}, that the results of \cite{STgmf}
are rather robust
with respect to the 
distribution of the $J_{i,j}$'s. 
In particular, we consider the case of independent $J_{i,j}$
with $\expect*{J_{i,j}} = 0$. If the tails of the distributions
of the $J_{i,j}$ are not too fat, then after proper normalization,
the value of the minimum of $H$ is close to its expectation
with high probability and 
does not depend on the actual 
distribution of the $J_{i,j}$. 
Further, with the help of numerical simulations
we provide some insight into the structure of both
the minimizer and the maximizer of $H$ and show that also
this structure is robust.
Note that 
our results hold in the diluted case
as well, that is in the case where each $J_{i,j}$ is zero
with a certain probability. This probability needs not to be fixed, 
but it is allowed to go to $1$ as $N \to \infty$ in a suitable way.

Rigorous results
are presented in Section~\ref{sec:main_results}, whereas
the numerical findings concerning the structure
of the minimizer and the maximizer of $H$
are detailed in Section~\ref{sec:numerical_results}.
Proofs are given in Section~\ref{sec:proofs}.

Throughout the paper we use the jargon of statistical
mechanics. As a consequence, we use expressions like
(particle) configuration when referring to $\eta$, number of
particles when referring to $N$, Hamiltonian for
the quadratic form $\eqref{eq:general_quadratic_form}$ and energy
and energy per particle (of a configuration $\eta$) for, respectively,  
$H(J, \eta)$ and $\frac{H(J, \eta)}{N}$.
Likewise, the minimizer \eqref{eq:general_quadratic_form} is
often referred to as the ground state of $H$ and we use
the expression thermodynamic limit to denote the limit
as $N \to \infty$.

\begin{remark}
    Note that $H$ denotes a family of random variables
    indexed by $N$. However, we do not
    write this dependence explicitly to lighten the notation.
\end{remark}

\section{Main results}\label{sec:main_results}

We consider random instances of QUBO, that is we assume
the matrix $J = \{J_{i,j}\}_{1 \le i,j \le N}$ to be the realization 
of some multivariate random variable.
Unless otherwise specified we will assume the $J_{i,j}$'s to be
independent, identically distributed, and such that $\expect*{J_{i,j}} = 0$ and 
$\variance{J_{i,j}} = \sigma^2$.
As for the value of the normalizing constant $W$ we take
it to be such that the random variable $\sum_{ij} J_{i,j}$
has variance $N$. 

\begin{remark}\label{thm:rem_mixture_distribution}
    Note that, in general, the distribution of the $J_{i,j}$
    is allowed to have atoms. 
    In particular, we will be interested in random variables taking the value
    zero with probability $1 - p(N) = 1 - N^{\delta - 2}$ for $\delta$
    in $\left(1, 2\right]$. In this way we can include in our
    analysis the diluted case, that is the case where the matrix $J$ of 
    the coefficients of the objective function is ``sparse''
    with expected density $\rho = p(N)$. Note that, in this way, the
    average degree of each vertex grows unbounded with $N$.
    For $\sum_{ij} J_{i,j}$ to have variance $N$ one should set
    $W = \frac{1}{\sqrt{\rho N \variance{J_{1,1}}}}$.
\end{remark}

In the remainder of the paper, we will use the following notation.

Let 
\begin{equation*}
    \eta^{\min} :=\argmin\limits_{\eta \in \{0, 1\}^N} H(J, \eta)
    \quad \text{ and } \quad
    \eta^{\max} :=\argmax\limits_{\eta \in \{0, 1\}^N} H(J, \eta),
\end{equation*}
that is $\eta^{\min}$ and $\eta^{\max}$ are, respectively, 
the minimizer and the maximizer of $H$.

Further, let 
\begin{equation*}
    \min_{\eta \in \{0, 1\}^N} H(J, \eta) =: -m_{\min, N} \cdot N
    \quad \text{ and } \quad
    \max_{\eta \in \{0, 1\}^N} H(J, \eta) =: m_{\max, N} \cdot N.
\end{equation*}
In words, $m_{\min, N}$ and $m_{\max, N}$ are, respectively,
the minimum and the maximum \emph{per particle} of $H$.

Moreover, setting $|\eta| = \sum_{i = 1}^N \eta_{i}$ we call 
\begin{equation*}
    \alpha_{\min, N} = \frac{ |\eta^{\min}| }{N}\;
    \text{ and }\;
    \alpha_{\max, N} = \frac{ |\eta^{\max}| }{N}
\end{equation*}
the proportion of \emph{ones} in the minimizer and maximizer of $H$.

We remark that 
$\eta^{\min}$, $\eta^{\max}$,
$m_{\min, N}$, $m_{\max, N}$, 
$\alpha_{\min, N}$ and $\alpha_{\max, N}$ 
are random variables and depend on the
realization of $J$, but we do not write 
this dependence explicitly to lighten the notation.

We are interested in the limiting behavior
of $m_{\min, N}$, $m_{\max, N}$, $\alpha_{\min, N}$ and $\alpha_{\max, N}$.
When the $J_{i,j}$ are Gaussian, 
the existence of the limits 
$m_{\min}$, $m_{\max}$, $\alpha_{\min}$, $\alpha_{\max}$
for the previous quantities
follows from Theorems~1.1 and~1.2 in \cite{panchenko2005generalized}.
Reasonable numerical estimates for the quantities of interest are $\alpha_{\min} = \alpha_{\min} \approx 0.42$  
and $\alpha_{\min} = \alpha_{\max} \approx 0.624$ (see Section~\ref{sec:numerical_results} below).

We show that, as $N$ gets larger, the fluctuations of both the minimum and maximum per particle of $H$ around their expected values vanish
provided some conditions on the tails of the $J_{i,j}$ are satisfied.
More precisely, 
we have the following
\begin{theorem}\label{thm:min_max_self_averaging}
    Let the $J_{i,j}$ be independent identically distributed 
    random variables with $\expect*{J_{1,1}} = 0$ and
    $\variance*{J_{1,1}} = 1$. Then,
    \begin{enumerate}[label=(\alph*)]
    \item\label{thm:min_max_self_averaging_conv_prob} 
        as $N \to \infty$, 
        \begin{gather}
            m_{\min, N} - \expect*{m_{\min, N}} \overset{\operatorname{P}}{\longrightarrow} 0
            \label{eq:convergence_of_minimum_general}\\
            m_{\max, N} - \expect*{m_{\max, N}} \overset{\operatorname{P}}{\longrightarrow} 0.
            \label{eq:convergence_of_maximum_general}
        \end{gather}
        where $\overset{p}{\longrightarrow}$ denotes convergence in probability.
    \end{enumerate}
    Further,
    \begin{enumerate}[resume, label=(\alph*)]
    \item\label{thm:min_max_self_averaging_conv_as} 
          If $\expect*{\abs*{J_{1,1}}^{3}} < \infty$ the convergence
          in \eqref{eq:convergence_of_minimum_general}
          and \eqref{eq:convergence_of_maximum_general}
          is almost sure.
    \item\label{thm:min_max_self_averaging_conv_expfast} 
        Suppose the $J_{i,j}$ are strictly subgaussian, that is, 
        \begin{equation*}\label{eq:def_strictly_subgaussian}
            \expect*{\exp(sJ_{1,1})} \leq \exp\left( \frac{s^{2}}{2}\right),
            \quad \forall s \in \mathbb{R}.
        \end{equation*}
        Then there exists a positive constant $D$ such that,
        for all $z > 0$
        \begin{equation*}
            \proba*{\abs[\big]{m_{\min, N} - \expect*{m_{\min, N}}} > Nz } \leq e^{-DNz}
            \,
            \text{ and }
            \,
            \proba*{\abs[\big]{m_{\max, N} - \expect*{m_{\max, N}}} > Nz } \leq e^{-DNz}
        \end{equation*}
        that is the convergence is \emph{exponentially fast}.        
    \end{enumerate}
\end{theorem}

The proof is provided in Section~\ref{sec:proof_thm_min_max_self_averaging}.

Moreover the expected value of the minimum and maximum
per particle
do not depend on the actual distribution of the random couplings, 
provided they have finite variance.

\begin{theorem}\label{thm:universality_min_max}
    Let $J = \left\{ J_{ij} \right\}_{1\leq i,j\leq N}$ 
    and  $Y = \left\{ Y_{i,j} \right\}_{1\leq i,j\leq N}$ 
    be two independent sequences of 
    independent random variables,  
    such that for every $i, j$ 
    $\expect*{ J_{i,j} } =
     \expect*{ Y_{i,j} } = 0$ and 
    $\expect*{J_{i,j}^{2}} = \expect*{Y_i^2} = 1$.
    Also, set 
    $\gamma := \max \left\{\expect*{\left|J_{ij}\right|^3} , 
                           \expect*{\left|Y_{ij}\right|^3} , 
                           1 \leq i,j \leq N\right\}$;
    $\gamma$ may be infinite.
    Then we have, as $N \to \infty$,
    \begin{gather}
        \frac{1}{N} \left|\expect*{\min_{\eta} H(J, \eta)} - 
                          \expect*{\min_{\eta} H(Y, \eta)}\right|  \to 0 \\
        \frac{1}{N} \left|\expect*{\max_{\eta} H(J, \eta)} - 
                          \expect*{\max_{\eta} H(Y, \eta)}\right|  \to 0                 
    \end{gather}
    If furthermore $\gamma < \infty$,
        \begin{gather}
        \frac{1}{N} \left|\expect*{\min_{\eta} H(J, \eta)} - 
                          \expect*{\min_{\eta} H(Y, \eta)}\right|  
                    \leq  C N^{-1/6} \\
        \frac{1}{N} \left|\expect*{\max_{\eta} H(J, \eta)} - 
                          \expect*{\max_{\eta} H(Y, \eta)}\right|  
              \leq  C N^{-1/6}                    
    \end{gather}
    where $C$ is a constant depending only on $\gamma$.     
\end{theorem}

The proof comes as a consequence of an analogous 
result on the universality of the free energy which is presented 
in Section~{\ref{sec:proof_thm_universality_min_max}}. 

As an immediate consequence of Theorem~\ref{thm:universality_min_max} and 
Theorem~\ref{thm:min_max_self_averaging}\ref{thm:min_max_self_averaging_conv_expfast},
the minimum and maximum per particle have the same almost sure lower bound of the Gaussian case (see \cite{STgmf})
irrespective of the details of the distribution of the $J_{i,j}$.
More precisely, let $\nu^{-}(m) = |\{\eta : H(J, \eta) \leq -m N\}|$ denote the
number of configurations whose energy is less or equal to $-m N$ and, similarly,
let $\nu^{+}(m) = |\{\eta : H(J, \eta) \ge m N\}|$ be the number of
configurations with energy at least equal to $m N$. Then:

\begin{corollary}\label{thm:almost_sure_bound_for_min_and_max}
    Let the $J_{i,j}$'s be independent identically distributed 
    strictly subgaussian random variables with
    $\expect*{J_{1,1}} = 0$ and
    $\variance*{J_{1,1}} = 1$ 
    and let 
    $I(x) = - x \log(x) - (1-x) \log(1 - x)$.
    Then, for large values of $N$ and for some constant $C$,
    $\proba*{\nu^{-}(m^{\star}) > 0} \leq e^{-CN}$ and
    $\proba*{\nu^{+}(m^{\star}) > 0} \leq e^{-CN}$
    where $m^{\star} \approx 0.562$ 
    is the extremal value of $m$ such that the function
    $I(\alpha) - \frac{m^2}{2\alpha^{2}(1 - \alpha)^{2}}$,  
    for fixed $m$, has no zeros. This value is 
    obtained for $\alpha = \alpha^{\star} \approx 0.644$. 
\end{corollary}

\section{Conjectures and numerical results}\label{sec:numerical_results}

In this section, we present some numerical results concerning the minimum and
maximum per particle and the structure of the minimizer and the maximizer of
QUBO instances with random coefficients. 
The findings presented in Section~\ref{sec:numerical_min_max_per_particle} 
allow us to highlight some interesting features concerning the connection between particles contributing to the minimizer of $H$ (that is components of $\eta^{\min}$ equal to $1$) and particles contributing to the maximizer of $H$ (that is components of $\eta^{\max}$ equal to $1$) and the probability of the events $\{\eta^{\min}_{i} = 1\}$ and $\{\eta^{\max}_{i} = 1\}$.
Simulations have been carried over using Monte Carlo Probabilistic Cellular Automata (PCA) as those introduced in
\cite{STgmf,apollonio2019criticality,d2021parallel,apollonio2022shaken,scoppola2022shaken3D}
The advantage of these algorithms is represented by their inherently parallel
nature which allows the exploitation of parallel architectures at their fullest while preserving a quality of the solution comparable to the one obtained with \emph{single spin flip} MCMC algorithms as outlined in
\cite{STgmf,apollonio2022shaken,fukushima2023mixing}. Thanks to these algorithms we were able, in the diluted case, to simulate effectively systems with $N$ up to $128000$.

In particular, the algorithm used in the simulations works as follows.
Let 
\begin{equation}\label{eq:double-hamiltonian}
    H(\eta, \tau) = \sum_{i,j} J_{ij}\eta_{i}\tau_{j} + q \sum_{i}(1 - \sigma_{i}\tau_{i})
\end{equation}

As a preliminary step consider the symmetrized version of $\tilde{J}$ of $J$ (that is $\tilde{J}=\frac{J + J^{T}}{2}$) and note that the value of the Hamiltonian dos not change if we replace $J$ with $\tilde{J}$.
Let $H(\eta, \tau) = \beta \sum_{i}h_{i(\eta)\tau_{i}} + q\sum_{i}\left[ \eta_{i}(1-\tau_{i}) + \tau_{i}(1-\eta_{i}) \right]$ with $h_{i}(\eta) =\frac{1}{\sqrt{ N }}\sum_{j}\tilde{J}_{i,j}\eta_{j}$ and $\beta$ and $q$ two positive parameters and choose the transition matrix to be $P_{\eta, \tau} = \frac{e^{ -H(\eta, \tau) }}{\sum_{\tau}e^{ -H(\eta,\tau) }}$.  Rewriting $P_{\eta, \tau}$ as
\begin{equation}    
P(\eta, \tau) = \prod_{i} \frac{e^{ -\beta h_{i}(\eta)\tau_{i}-q[\eta_{i}(1-\tau_{i}) + \tau_{i}(1-\eta_{i})] }}{Z_{i}}
\end{equation}
we immediately see that, conditionally on $\eta$, the value of  each $\tau_{i}$ can be sampled independently with the following probabilities:
\begin{equation}
    P(\tau_{i}=1) = \frac{e^{ -\beta h_{i}(\eta) - q \eta_{i} }}{Z_{i}} \; \text{ and } \; P(\tau_{i}=0|\eta) = \frac{e^{ -q (1-\eta_{i}) }}{Z_{i}}
\end{equation}
with $Z_{i} = e^{ -\beta h_{i} + q \eta_{i} } + e^{ - q(1-\eta_{i}) }$.
Then, at least in principle, each component of of the configuration could be updated on a dedicated \emph{core}.

Note that, by the symmetry of $\tilde{J}$ we have that $\frac{\sum_{\eta}e^{ -H(\eta,\tau) }}{\sum_{\eta,\tau}e^{ -H(\eta,\tau) }}$ is the reversible measure of $P_{\eta,\tau}$. Further observe that $H(\eta, \eta) = \beta H(\eta)$. 
The intuitive rationale why the heuristic algorithm is expected to work is the following. As $q$ gets large, the weight of pairs $(\eta, \tau)$ with $\eta \neq \tau$ in $\pi(\sigma)$ becomes smaller and the stationary measure of the algorithm becomes closer to the Gibbs measure $\frac{e^{ -\beta H(\eta) }}{Z}$ with Hamiltonian $H$ and inverse temperature $\beta$. A comparison between this algorithm and the Metropolis one is provided in \cite{STgmf}.

The computation of the local fields $h_{i}$ proves to be the computationally expensive part of the algorithm. However the vector of local fields
$\{ h_{i}(\eta) \}_{i = 1, \dots, N}$ is the matrix-vector product $J \cdot \eta$ and can be computed using highly optimized linear algebra libraries and exploiting a parallel computing device such as a GPU.
Further observe that $k$ simulations can be carried over in parallel, possibly with different parameters $\beta$ and $q$. Let $\vec{E}$ be the $N \times k$ matrix whose columns contain the $k$ configurations to be updated. In this case the all collection of vectors of fields $\{ h_{i}(\eta^{(m)} ) \}_{i = 1, \dots, N; m = 1, \dots, k}$ (encoded, again, as an $N \times k$ matrix) is the matrix-vector product $J \cdot \vec{E}$. Also this product can be computed effectively using highly optimized libraries and is, in general, substantially faster than computing $k$ matrix vector products.

We remark that the values we found are heuristic and we have no guarantee that 
they coincide with the true values of the minimimizer and the maximizer of $H(J)$. 
However, as already discussed in more details in \cite{STgmf} the exact algorithm introduced in \cite{rendl2010solving} can be exploited to compute the exact minima and maxima of $H(J)$ for values of $N$ up to 150. The absolute value of the minimum and the maximum per particle of $H(J)$ are steadily close to $0.42$ already for values of $N$ between 80 and 150.
Moreover, for instances with sizes up to $150$, we verified that our algorithm typically finds the same optimum as the exact algorithm. 
For the larger instances, we retained the minimum and maximum found in a set of runs of the algorithms with different pairs of values of $\beta$ and $q$. 
\subsection{Minimum and maximum per particle}
\label{sec:numerical_min_max_per_particle}
The standard normal case has been investigated extensively in \cite{STgmf}:
for values of $N$ relatively small, both $m_{\min, N}$
and $m_{\max, N}$ oscillate around a value about $0.42$ whereas $\alpha_{\min,
N}$ and $\alpha_{\max, N}$ very rapidly approach a value about $0.624$
(see Fig.~\ref{fig:m_and_alpha_N_standard_gaussian}).
Theorem \ref{thm:universality_min_max} states that both $m_{\min, N}$ and
$m_{\max, N}$ are robust with respect to the distribution of the $J_{i,j}$.
Numerical simulations show that this robustness also concerns the proportion
of ones in both the minimizer and maximizer of $H$.

Fig.~\ref{fig:m_and_alpha_N_exponential} shows the behavior of 
$m_{\min, N}, m_{\max, N}, \alpha_{\min, N}, \alpha_{\max, N}$ for
$J_{i,j}$ with (shifted) exponential distribution. Even if the
exponential distribution is rather asymmetric and is not subgaussian, 
for values of $N$ relatively small the average energy per particle 
and the proportion of ones in both the minimizer and the maximizer
approach those of the standard normal case.
\begin{figure}[H]
    \centering
        \begin{subfigure}[c]{0.78\textwidth}
            \includegraphics[width=\textwidth]{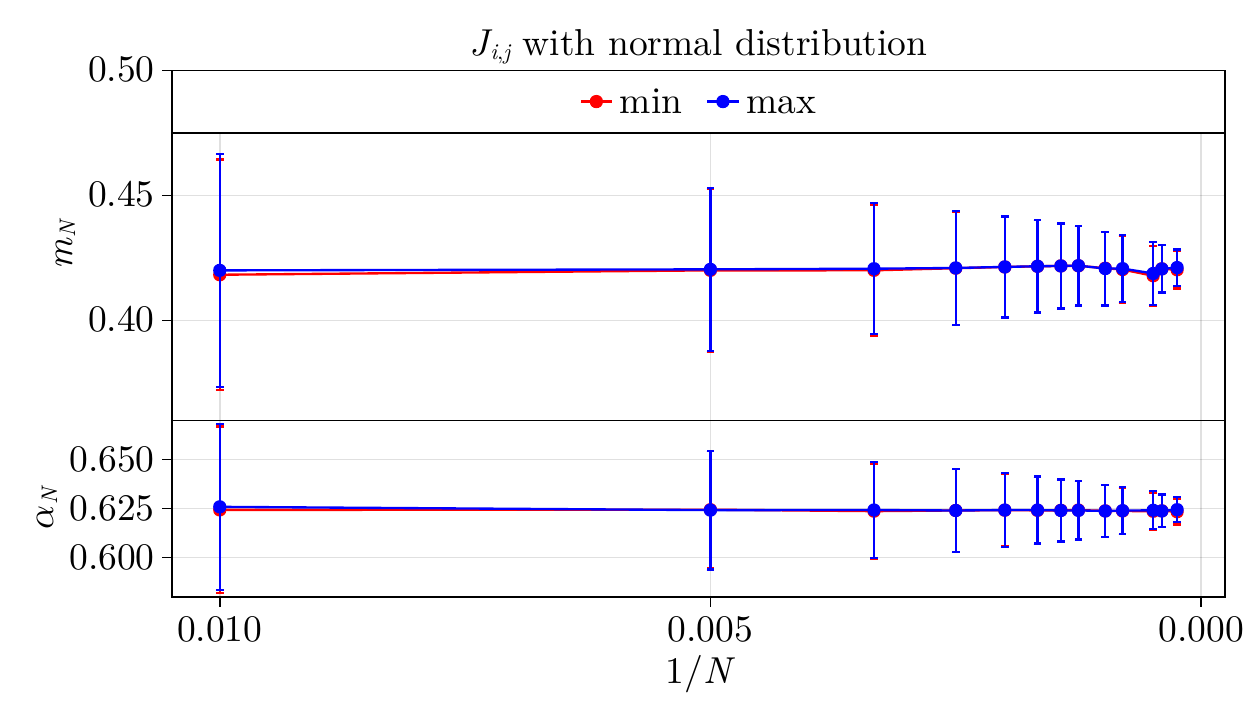}
            \caption{}
        \end{subfigure}
        \hfill
        \begin{subtable}[c]{0.19\textwidth}            
            \begin{tabular}{@{}lll@{}}
                \toprule
                $N$  & $m_N$ & $\alpha_N$ \\ \midrule
                100  & 0.419 & 0.625      \\
                200  & 0.42  & 0.624      \\
                300  & 0.42  & 0.624      \\
                400  & 0.421 & 0.624      \\
                500  & 0.421 & 0.624      \\
                600  & 0.422 & 0.624      \\
                700  & 0.422 & 0.624      \\
                800  & 0.422 & 0.624      \\
                1024 & 0.421 & 0.624      \\
                1250 & 0.42  & 0.624      \\
                2048 & 0.417 & 0.624      \\
                2500 & 0.421 & 0.624      \\
                4096 & 0.421 & 0.624      \\ \bottomrule
            \end{tabular}
            \caption{}
            \label{tab:gaussian_estimates}
        \end{subtable}
        \caption{Average values of $m_{\min, N}$, $m_{\max, N}$,
                 $\alpha_{\min}$ and $\alpha_{\max, N}$
                 in the case of
                 standard normally distributed $J_{i,j}$'s.
                 Values of $m_N$ and $\alpha_N$ in the table
                 are computed as
                 averages of $m_{\min, N}, M_{\max, N}$ and
                 $\alpha_{\min}, \alpha_{\max, N}$ respectively.
                 The length of each branch of the error bars is equal to to standard error of the empirical average. Averages are computed over 10000 realization of $J$ for $N$ up to $1024$ and between $1000$ and $5000$ realization of $J$ for larger instances}
        \label{fig:m_and_alpha_N_standard_gaussian}
\end{figure}
\begin{figure}[H]
    \begin{center}
        \includegraphics[width=0.78\textwidth]{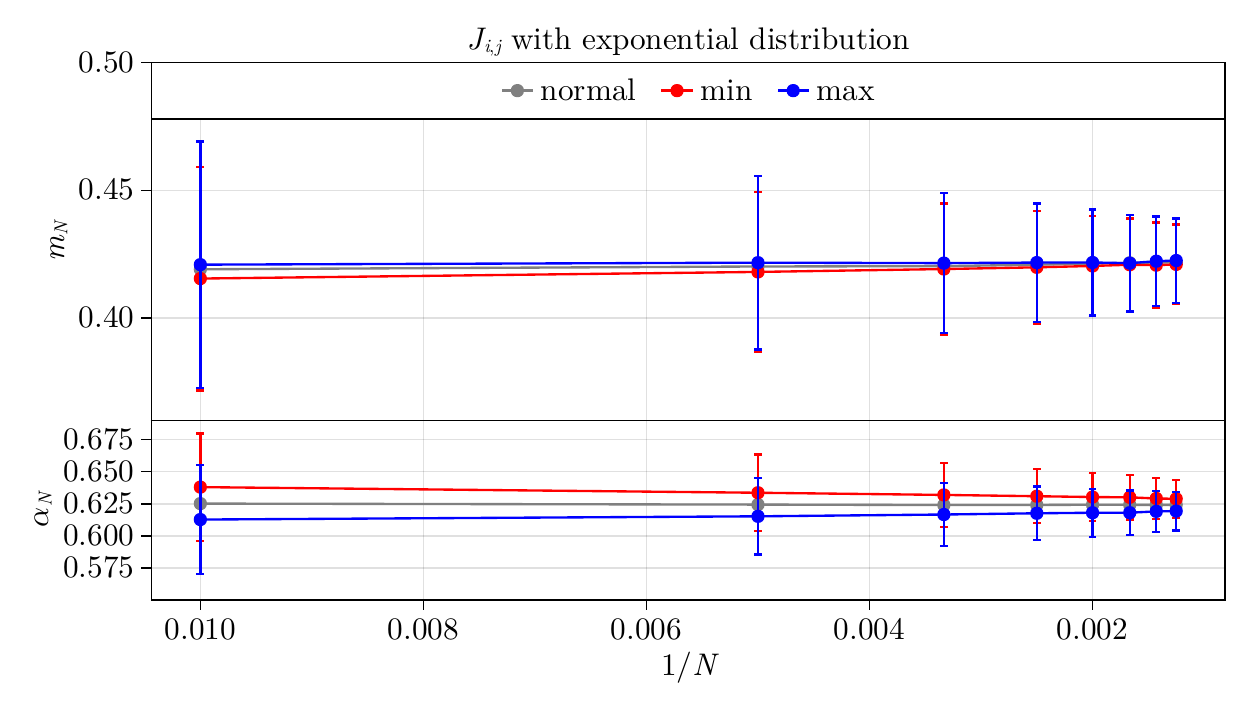}
    \end{center}
    \caption{In this case each $J_{i,j}$ is distributed as $X - 1$ with $X$ an exponential random variable with expected value $1$. The exponential distribution is rather asymmetric and is not subgaussian. However, already for $N$ of order ``a few hundred'' the values of the minimum and maximum per particle of $H$ and the proportion of ``ones'' in the minimizer and the maximizer of $H$ approach those of the normal case. The length of each branch of the error bars is equal to to standard error of the empirical average (only shown for the exponential distribution). Averages are computed over 10000 realizations of $J$
    The curves for the normal distributions in this chart are the averages of the values  $m_{\min, N}$ and $m_{\max, N}$ in the first panel and $\alpha_{\min, N}$ and $\alpha_{\max, N}$ in the second panel.}
    \label{fig:m_and_alpha_N_exponential}
\end{figure}
\begin{figure}[H]
    \begin{center}
        \includegraphics[width=0.78\textwidth]{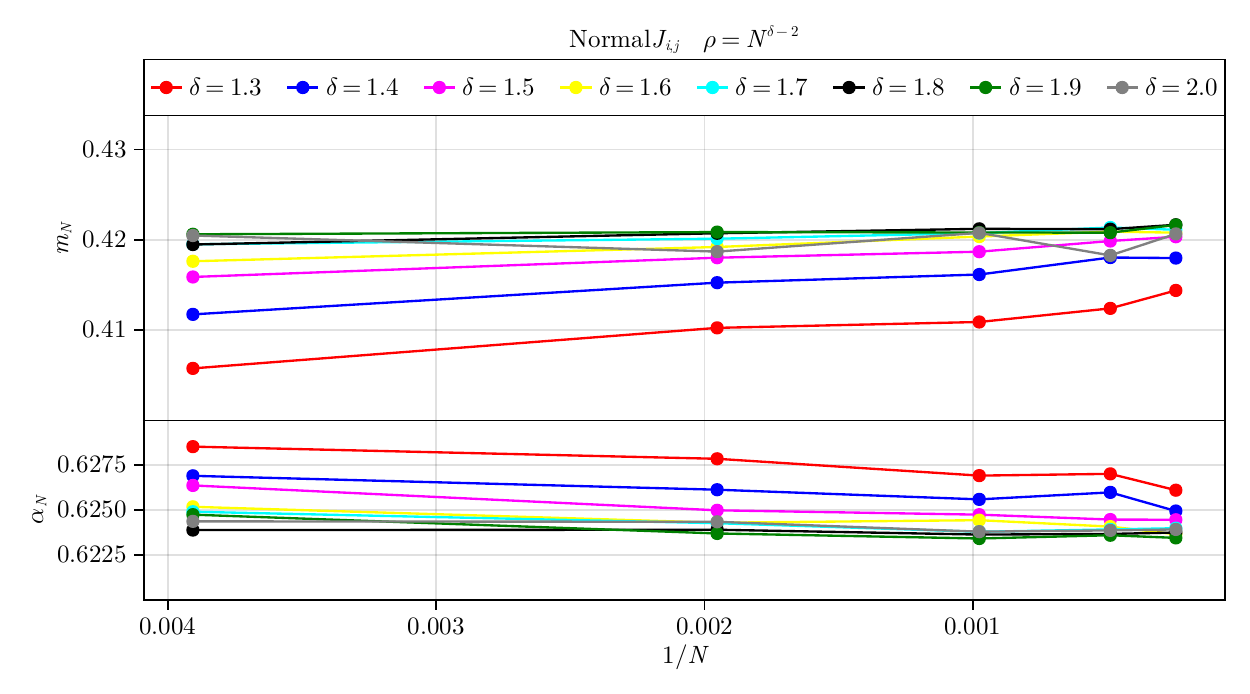}
    \end{center}
    \caption{Comparison (top panel) of the average of the minimum 
             and maximum \emph{per particle} of $H$ in the diluted case with the corresponding value in the standard normal case for several values of $\delta$. In all these cases the average value of the minimum and maximum per particle appear to approach the same limit (about $0.42$). As for the values of $\alpha_{N}$ (bottom panel), these appear to be very close to the value of $\alpha_{N}$ of the standard normal case. Averages, for $\delta < 2$, are computed over $100$ realization of $J$.}
    \label{fig:m_and_alpha_N_diluted_case}
\end{figure}

In our simulations, we also considered the \emph{diluted}
case, that is we took $J_{i,j}$ to be zero with probability
$1 - p_{\delta}(N)$ and a standard normal random variable
otherwise where $p_{\delta}(N) = N^{\delta - 2}$.
Findings concerning the behavior of 
$m_{\min, N}, m_{\max, N}, \alpha_{\min, N}, \alpha_{\max, N}$ 
are presented in
Fig.~\ref{fig:m_and_alpha_N_diluted_case} and
Fig.~\ref{fig:m_and_alpha_N_highly_diluted_case}.
The log-log plot of Fig.~\ref{fig:m_convergence_diluted_case}
suggests that $m_N$ converges to $\bar{m}$ as a power
of $N$.
\begin{figure}[H]
    \begin{center}
        \includegraphics[width=0.78\textwidth]{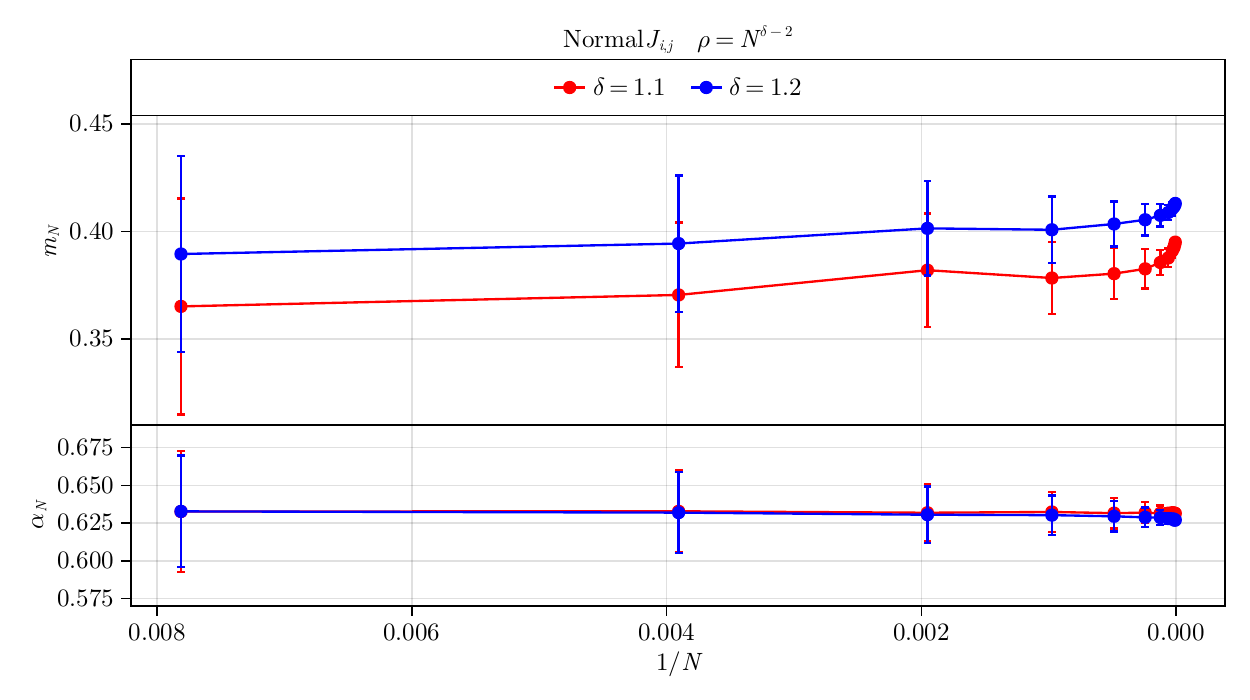}
    \end{center}
    \caption{Average of the minimum per particle and of $\alpha_{\min, N}$
            in the diluted case for $\delta = 1.2$ and $\delta = 1.1$.
            The length of each branch of the error bars is equal to standard error of the empirical average.
            Averages are computed over $900$ realizations of $J$ for
            $N$ up to $1024$, $100$ realizations for $N$ between $2048$
            and $32000$ and 36 realizations for larger values of $N$.
            }
    \label{fig:m_and_alpha_N_highly_diluted_case}
\end{figure}

\begin{figure}[H]
    \begin{center}
        \includegraphics[width=0.78\textwidth]{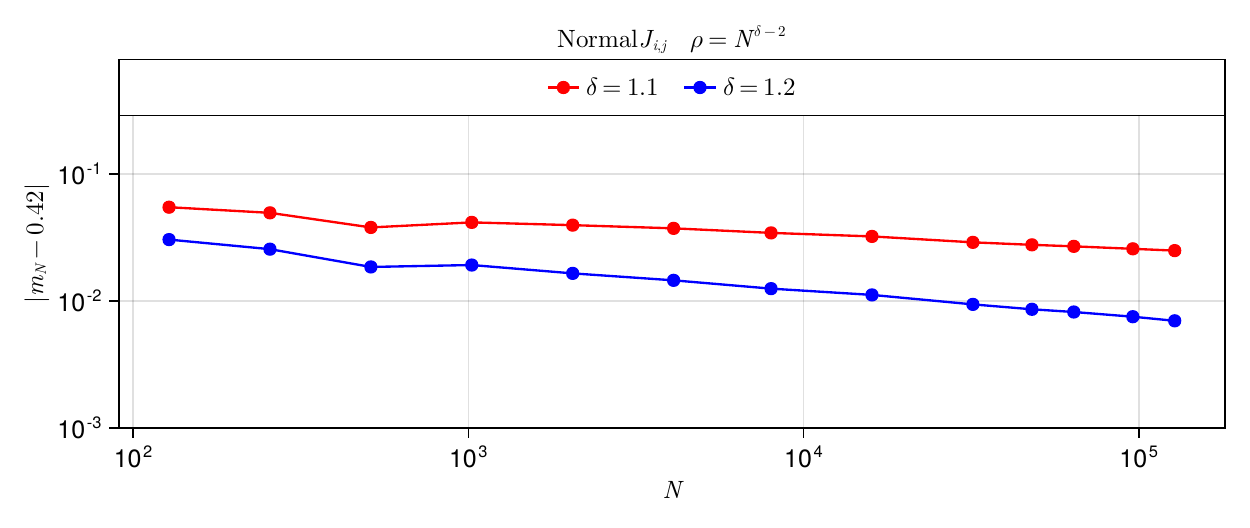}
    \end{center}
    \caption{log-log plot of the distance of the absolute value of the minimum per particle from the conjectured limiting value $\bar{m}$ as a function of $N$ in the diluted case for $\delta = 1.2$ and $\delta = 1.1$.
    For both $\delta = 1.2$ and $\delta = 1.1$ the lines appear to have a negative slope suggesting that, in both cases, the absolute value of the minimum per particle of $H$ will reach the value $\bar{m}$}
    \label{fig:m_convergence_diluted_case}
\end{figure}

We highlight that as $\delta$ becomes smaller, the density
of $J$ becomes \emph{very small}. To give an idea, values of the
density of $J$ are given in
Table~\ref{tab:densities}. For instance, with $N = 128000$ and $\delta = 1.1$ 
the expected number of nonzero entries in each row of $J$ is less than $4$. 
\begin{table}[H]
    \centering
        \begin{tabular}{@{}lccccccccccc@{}}
            \toprule
            N & 4000 & 4000 & 4000 &  8000 & 8000 & 8000  & 16000 & 16000 & 16000 & 128000  & 128000\\ 
            \midrule
            $\delta$ & 1.9 & 1.8 & 1.3 & 1.9 & 1.8 & 1.3 & 1.9 & 1.8 & 1.3 & 1.3 & 1.1\\
            \midrule
            $\rho$ & 0.4363 & 0.1904 & 0.003 & 0.4071 & 0.1657 & 0.0019 & 0.3798 & 0.1443 & 0.0011 & 0.0003 & \num{3e-5}\\ 
            \bottomrule
        \end{tabular}
        \caption{Density of matrix $J$ for several values of $N$ and $\delta$}
        \label{tab:densities}
\end{table}

\begin{table}[H]
    \centering
    \begin{tabular}{@{}lccSS@{}}
        \toprule
        Instance id & \multicolumn{1}{c}{optimum in \cite{wang2013probabilistic}} & \multicolumn{1}{c}{N} & \multicolumn{1}{c}{$\rho$} & \multicolumn{1}{c}{$m_N$} \\ \midrule
        p3000.1  & 3931583  & 3000 & 0.5 & 0.412 \\
        p3000.2  & 5193073  & 3000 & 0.8 & 0.431 \\
        p3000.3  & 5111533  & 3000 & 0.8 & 0.424 \\
        p3000.4  & 5761822  & 3000 & 1   & 0.427 \\
        p3000.5  & 5675625  & 3000 & 1   & 0.421 \\
        p4000.1  & 6181830  & 4000 & 0.5 & 0.421 \\
        p4000.2  & 7801355  & 4000 & 0.8 & 0.42  \\
        p4000.3  & 7741685  & 4000 & 0.8 & 0.417 \\
        p4000.4  & 8711822  & 4000 & 1   & 0.42  \\
        p4000.5  & 8908979  & 4000 & 1   & 0.429 \\
        p5000.1  & 8559680  & 5000 & 0.5 & 0.417 \\
        p5000.2  & 10836019 & 5000 & 0.8 & 0.418 \\
        p5000.3  & 10489137 & 5000 & 0.8 & 0.404 \\
        p5000.4  & 12252318 & 5000 & 1   & 0.422 \\
        p5000.5  & 12731803 & 5000 & 1   & 0.439 \\
        p6000.1  & 11384976 & 6000 & 0.5 & 0.422 \\
        p6000.2  & 14333855 & 6000 & 0.8 & 0.42  \\
        p6000.3  & 16132915 & 6000 & 1   & 0.423 \\
        p7000.1  & 14478676 & 7000 & 0.5 & 0.426 \\
        p7000.2  & 18249948 & 7000 & 0.8 & 0.425 \\
        p7000.3  & 20446407 & 7000 & 1   & 0.425 \\ \bottomrule
    \end{tabular}
    \caption{Values of the maximum per particle for some benchmark 
             instances. In these instances $J_{i,j}$ 
             drawn uniformly at random from the integers
             in $[-100, 100]$ and the matrix $J$ is symmetric.
             Consequently, to compare the values with
             those of the standard normal case, 
             the normalizing constant $W$ appearing
             in \eqref{eq:general_quadratic_form} must
             be set equal to $\sqrt{\frac{6}{\rho N (201^{2} - 1)}}$.
             The \emph{optima} used to compute $m_N$ 
             are the best-known solutions reported in 
             \cite{wang2013probabilistic}}
    \label{tab:m_N_uniform_benchmarks}
\end{table}
In the operation research literature, to test optimization algorithms,
it is common to consider
random instances of QUBO where the $J_{i,j}$ have a uniform
distribution (see, e.g., \cite{wang2013probabilistic,alidaee2017simple,waidyasooriya2020gpu,liang2022data})
and where the matrix $J$ is, possibly, \emph{sparse}.
Values of the best-known maximizer for some benchmark 
QUBO instances in the case of uniformly distributed $J_{i,j}$
are reported in Table~\ref{tab:m_N_uniform_benchmarks}.
It is apparent that also in these cases, the values of the optima
per particle agree with those of the standard normal case.

\subsection{Structure of minimizer and maximizer}
\label{sec:structure_min_max}
Observe that for any realization of $J$
the minimizer and the maximizer of the Hamiltonian are with probability 1 since the distribution of the couplings is absolutely continuous.
Then it is possible to partition the indices $1, 2, \ldots, N$
into four sets: 
\begin{itemize}[label=\textendash]
    \item $I_{1} = \{i : \eta^{\min}_{i} = 1, \; \eta^{\max}_{i} = 0\}$;
    \item $I_{2} = \{i : \eta^{\min}_{i} = 1, \; \eta^{\max}_{i} = 1\}$; 
    \item $I_{3} = \{i : \eta^{\min}_{i} = 0, \; \eta^{\max}_{i} = 1\}$;
    \item $I_{4} = \{i : \eta^{\min}_{i} = 0, \; \eta^{\max}_{i} = 0\}$.
\end{itemize}
To refer properly to the cardinality of these sets we give the following
\begin{definition}\label{def:alpha_i}
    $\alpha_{k,N} := \frac{ |I_{k}| }{ N }, \, k = 1, 2, 3, 4$.    
\end{definition}
With an abuse of notation, we say that the $i$-th row (column) 
of $J$ belongs to $I_{k}$ if $i \in I_{k}$.
Note that $\alpha_{1,N}$ can be interpreted as the 
proportion of $1$ appearing in the minimizer but not
appearing in the maximizer of $H$. Similar interpretations
can be given for $\alpha_{2,N}$, $\alpha_{3,N}$ and $\alpha_{4,N}$.
\begin{remark}\label{thm:alpha_as_sum_of_alpha_i}
    With the definition of $\alpha_{i}$ given above it is immediate
    to see that 
    $\alpha_{\min,N} = \alpha_{N,1} + \alpha_{N,2}$ and 
    $\alpha_{\max,N} = \alpha_{N,2} + \alpha_{N,3}$
\end{remark}

Leveraging on the definition of $I_{k}$, it is possible to
partition $J$ into 16 blocks 
\[
    J[k,\ell] := \{J_{i,j}, \; i \in I_{k}, \, j \in I_{\ell}\}
\]
Clearly, $J[I_{k}, I_{\ell}]$ is a sub matrix of $J$ with
$N\alpha_{k,N}$ rows and $N\alpha_{\ell,N}$ columns.

Numerical simulations suggest that the average value and the variance 
of the entries in each block, subject to proper normalization and
the relative size of the blocks 
converge to a deterministic limit as $N \to \infty$.
These limits do not depend on the distribution of the $J_{i,j}$'s,
as long as they have expected value zero and variance $1$.
More precisely we make the following
\begin{conjecture}\label{thm:blocks_features}
    Let $J_{i,j}$ be independent identically distributed random
    variables with $\expect*{J_{1,1}} = 0$ and
    $\variance*{J_{1,1}} = 1$. Then there exist constants
    $\alpha_{1}, \alpha_{2}, \alpha_{3}, \alpha_{4}$ such that
    \begin{enumerate}
        \item $\lim\limits_{N\to\infty} \alpha_{1,N} = \alpha_{1}$,\phantom{x}
              $\lim\limits_{N\to\infty} \alpha_{2,N} = \alpha_{2}$ \phantom{x}
              $\lim\limits_{N\to\infty} \alpha_{3,N} = \alpha_{3}$ \phantom{x} 
              $\lim\limits_{N\to\infty} \alpha_{4,N} = \alpha_{4}$
        \item $\alpha_{1} = \alpha_{2} = \alpha_{3} 
                    = \frac{1}{2}\alpha_{\min}
                    = \frac{1}{2}\alpha_{\max}$; 
                    \phantom{x} $\alpha_{4} = 1 - \frac{3}{2}\alpha_{\min}$ 
        \item $\variance*{J[k, \ell]_{i,j}} = 1$ for $k, \ell = 1, \ldots, 4$
    \end{enumerate}
\end{conjecture}
Further, set $\mu[k, \ell] = \expect*{J[k,\ell]_{i,j}}$ and write the covariance of
$J[k,\ell]_{i}, J[k, \ell]_{j}$ as 
\[\covariance*{J[k,\ell]_{i}, J[k, \ell]_{j}} 
= \frac{N^{2}\alpha_{k}\alpha_{\ell}}{2}(1 - \tilde{\sigma}[k,\ell])\].

Computation of averages over 10000 realizations of $J$ with standard
normal distribution and $N = 1024$ yielded the following estimates 
for $\alpha_{i}$, $\mu[k,\ell]$ and $\tilde{\sigma}[k,\ell]$.
\begin{gather}
    \vec{\alpha} = \begin{bmatrix}\begin{tabular}{@{} S[table-format=2.6] @{}}
    0.3105 \\
    0.31331 \\
    0.31054 \\
    0.06565 \\
    \end{tabular}
    \end{bmatrix}
    \\
    \mu[k,\ell] = \begin{bmatrix}\begin{tabular}{@{}*{4}{S[table-format=3.5]}@{}}
        -0.02161 & -0.05687 & 0.00006 & 0.04229 \\
        -0.05687 & 0.00001 & 0.05684 & 0.00013 \\
        -0.00005 & 0.05685 & 0.02161 & -0.04241 \\
        0.04228 & 0.00003 & -0.04227 & 0.00001 \\
        \end{tabular}
    \end{bmatrix}
    \\
    \tilde{\sigma}[k,\ell] = \begin{bmatrix}\begin{tabular}{@{}*{4}{S[table-format=2.6]}@{}}
        0.92996 & 1.11006 & 0.74334 & 1.20509 \\
        1.06221 & 0.58133 & 1.08835 & 0.77939 \\
        0.73879 & 1.08343 & 0.9415 & 1.25195 \\
        1.21106 & 0.76959 & 1.2258 & 0.99425 \\
    \end{tabular}
\end{bmatrix}
\end{gather}
These values suggest that the random couplings in blocks 
$[2, 2]$, $[3, 1]$, $[4, 2]$, $[1, 3]$,  and $[2, 4]$ 
are negatively correlated,
the random couplings in blocks
$[1, 4]$, $[3, 4]$, $[4, 1]$, and $[4, 3]$ are positively correlated,
whereas the random couplings in the remaining blocks
are roughly independent.

Very similar values can be obtained in the diluted case and for $J_{i,j}$
with distributions other than the normal.

Note that it is always possible to
relabel the indices $1, 2, \ldots, N$
so that, in $J$, indices in  $I_{1}$ appear ``first'', 
indices in $I_{2}$ appear  ``second'', 
in $I_{3}$ appear ``third'' and 
indices in $I_{4}$ appear ``last''.
With this relabelling, a graphical representation of 
Conjecture~\ref{thm:blocks_features} is provided in Fig.~\ref{fig:napkin01}.

\begin{figure}[H]
    \centering
        \begin{subfigure}[t]{0.23\textwidth}
            \includegraphics[width=\textwidth]{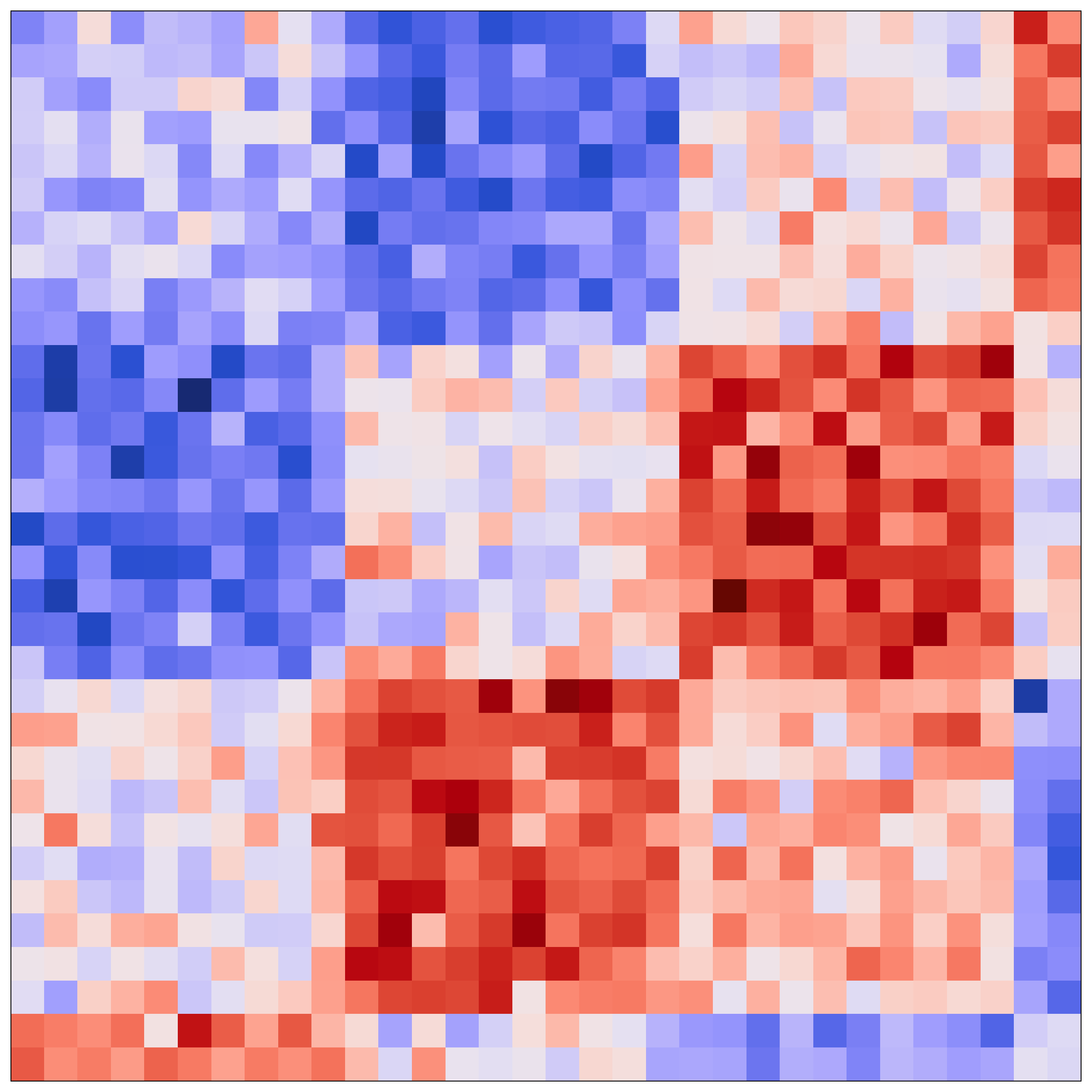}
            \caption{$\delta = 2; \; N = 4096, \; L = 128 $ \; $J_{i,j}$ normally distributed}
        \end{subfigure}
        \hfill
        \begin{subfigure}[t]{0.23\textwidth}
            \includegraphics[width=\textwidth]{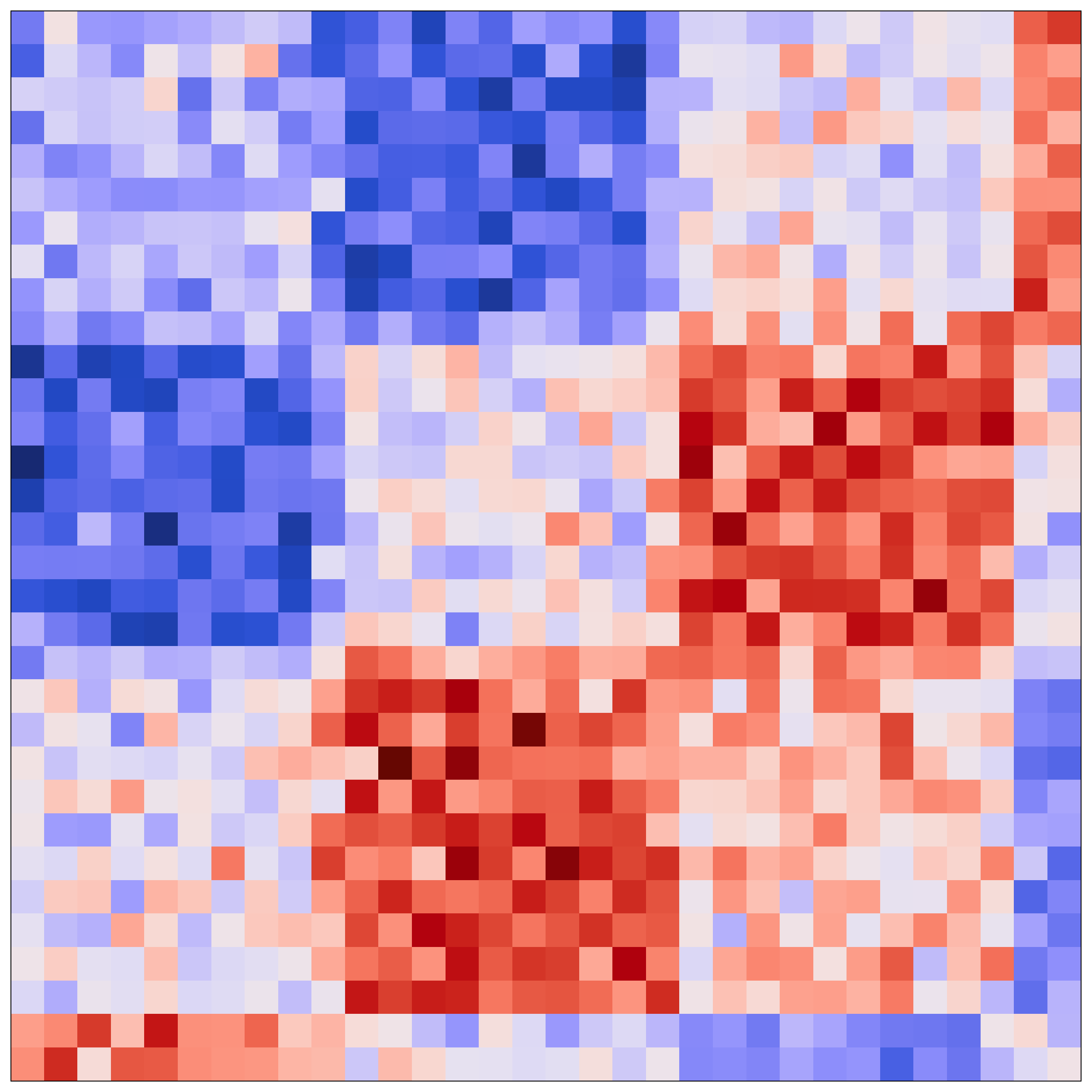}
            \caption{$\delta = 1.4 \; (\rho \approx 0.007); N = 4096, \; L = 128 $
                    $J_{i,j}$ normally distributed}
        \end{subfigure}
        \hfill
        \begin{subfigure}[t]{0.23\textwidth}
            \includegraphics[width=\textwidth]{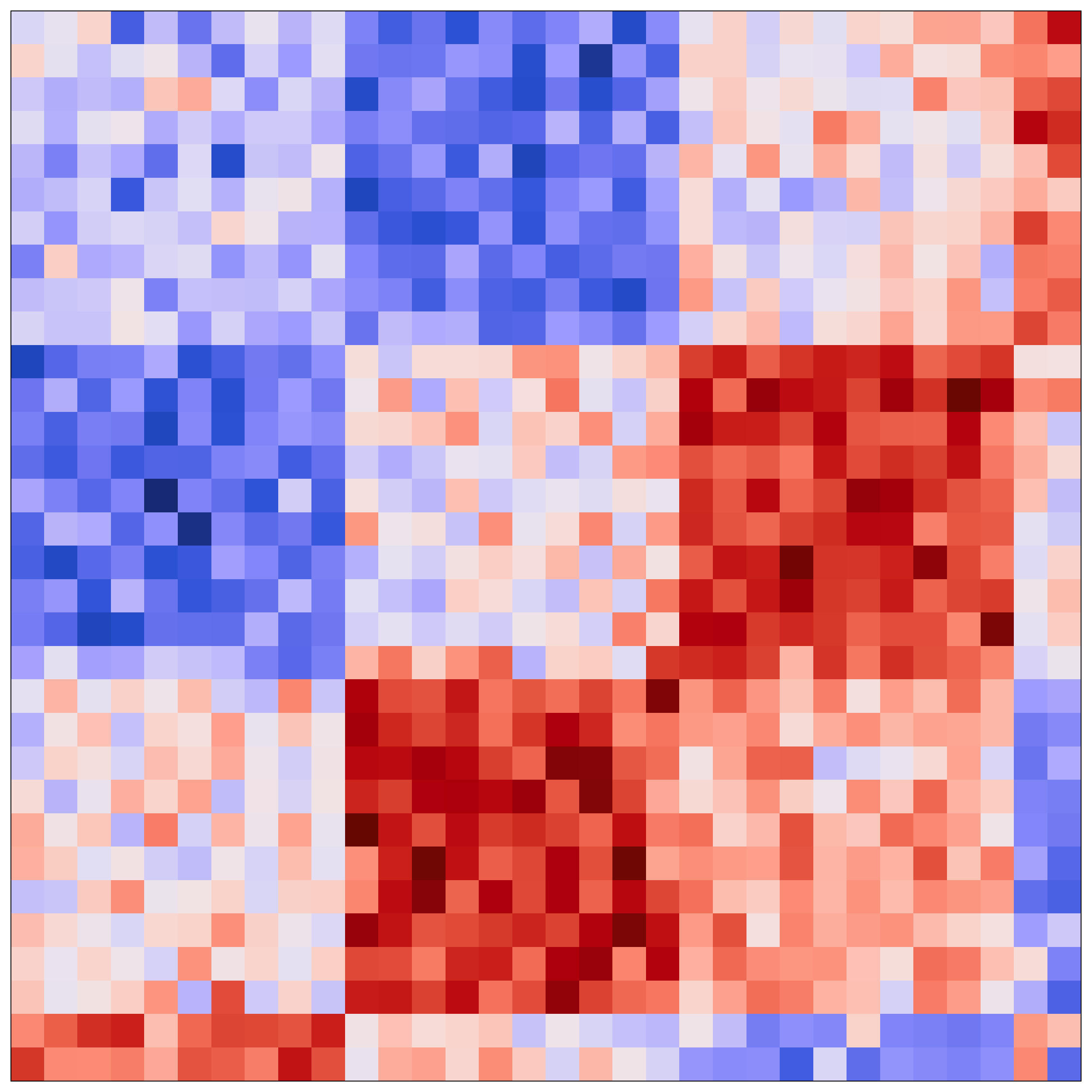}
            \caption{$\delta = 2; \; N = 4096, \; L = 128$  $J_{i,j}$ exponentially distributed}
        \end{subfigure}
        \hfill
        \begin{subfigure}[t]{0.23\textwidth}
            \includegraphics[width=\textwidth]{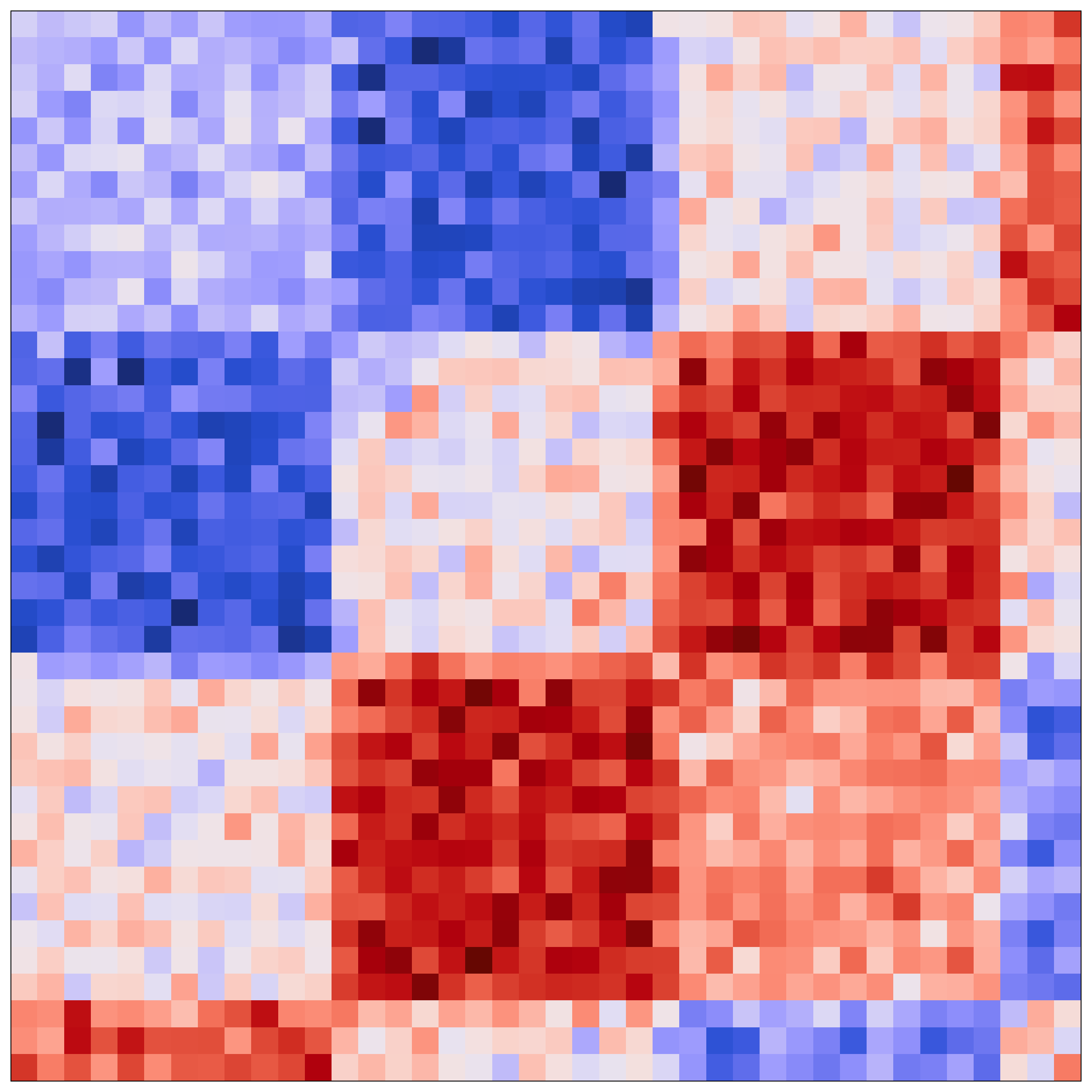}
            \caption{$\delta = 2; \; N = 7000, \; L = 175$  $J$ = p7000}
        \end{subfigure}
        \caption{Graphical representation of the matrix of the
        couplings. These matrices are obtained from $J$ by rearranging
        the rows (columns) of $J$ so that rows (columns) in $I_{1}$
        appear ``first'' rows (columns) $I_{2}$ appear ``second'',
        rows (columns) $I_{3}$ appear ``third'' and rows (columns)
        $I_{4}$ appear ``last''. Pictures are obtained by averaging
        the values of the $J_{i,j}$ over squares of size $L$. Negative
        (average) values of $J$ are blue whereas positive values are
        red; darker colors correspond to larger absolute values.}
        \label{fig:napkin01}
\end{figure}

\subsection{Probability a particle belongs to the minimizer/maximizer of $H$}\label{sec:prob_i_in_min_max}

Sets $I_{1}, \ldots I_{4}$ introduced above are random
sets depending on the realization of $J$.
The problem of finding $\eta^{\min}$ and $\eta^{\max}$
can be restated as the problem of determining
the sets $I_{\min} = I_{1} \cup I_{2}$ and
$I_{\max} = I_{2} \cup I_{3}$.
Though, as already mentioned, this problem is NP-hard,
we tried to assess, numerically, whether it is possible to
determine the probability that a certain index $i$
belongs to either $I_{\min}$ or $I_{\max}$.

To this aim, consider the matrix $\tilde{J} = \frac{J + J'}{2}$,
that is the symmetrized version of $J$ and the function
$\tilde{H}$ obtained from $H$ by replacing $J$ with $\tilde{J}$.
Note that $H(J, \eta) = \tilde{H}(\eta)$ for all $\eta$ and,
consequently, the values of $\eta^{\min}_{i}$ and $\eta^{\max}_{i}$ do not
depend on whether matrix ${J}$ or $\tilde{J}$ is considered. 

Take the matrix $\tilde{J}$ and define
$R = \sum_{j = 1}^{N} J_{i, j}$, that is
$R$ is the vector of the sums of the rows of $\tilde{J}$.
Let $[i]$ be the index of the $i$-th smallest element of $R$.
We say that the $[i]$-th row (column) of $\tilde{J}$
belongs to the minimizer (respectively the maximizer)
of $H$ if $\eta^{\min}_{[i]} = 1$ (respectively $\eta^{\max}_{[i]} = 1$).

Numerical simulations show that the probability that the
$[i]$-th row of $\tilde{J}$ belongs to the minimizer (respectively maximizer)
of $H$ is a decreasing (respectively increasing) deterministic  function of $\frac{[i]}{N}$.
This function is expected not to depend on the actual distribution of $J$.
Further, we expect that a positive fraction of the smallest (largest)
rows of $\tilde{J}$ (where smallest/largest refers to the sum of the elements
on each row of $\tilde{J}$) to belong to the minimizer (maximizer)
of $H$ with positive probability (see Fig.~\ref{fig:ith_row_in_etamax}).
More precisely we have the following
\begin{conjecture}\label{thm:prob_row_in_min_max}
    As $N \to \infty$, 
    $\proba*{\eta^{\min}_{[i]} = 1 } = f_{\min}\left(\frac{[i]}{N}\right) + o(1)$
    with $f_{\min}$ a decreasing function.
    Similarly, as $N \to \infty$, $\proba*{\eta^{\min}_{[i]} = 1 } = f_{\max}\left(\frac{[i]}{N}\right) + o(1)$ as , with
    $f_{\max}\left(\frac{[i]}{N}\right) = f_{\min}\left(1 - \frac{[i]}{N}\right)$.
    Moreover, there exists $\lambda_0 > 0$ such that $f_{\min} (\lambda) = 1$ for all
    $\lambda < \lambda_0$.
    Finally, if the tails of $J_{i,j}$ decay sufficiently fast,
    the function $f$ does not depend on the distribution of $J_{i,j}$
\end{conjecture}

\begin{figure}[H]
    \centering
            \includegraphics[width=0.78\textwidth]{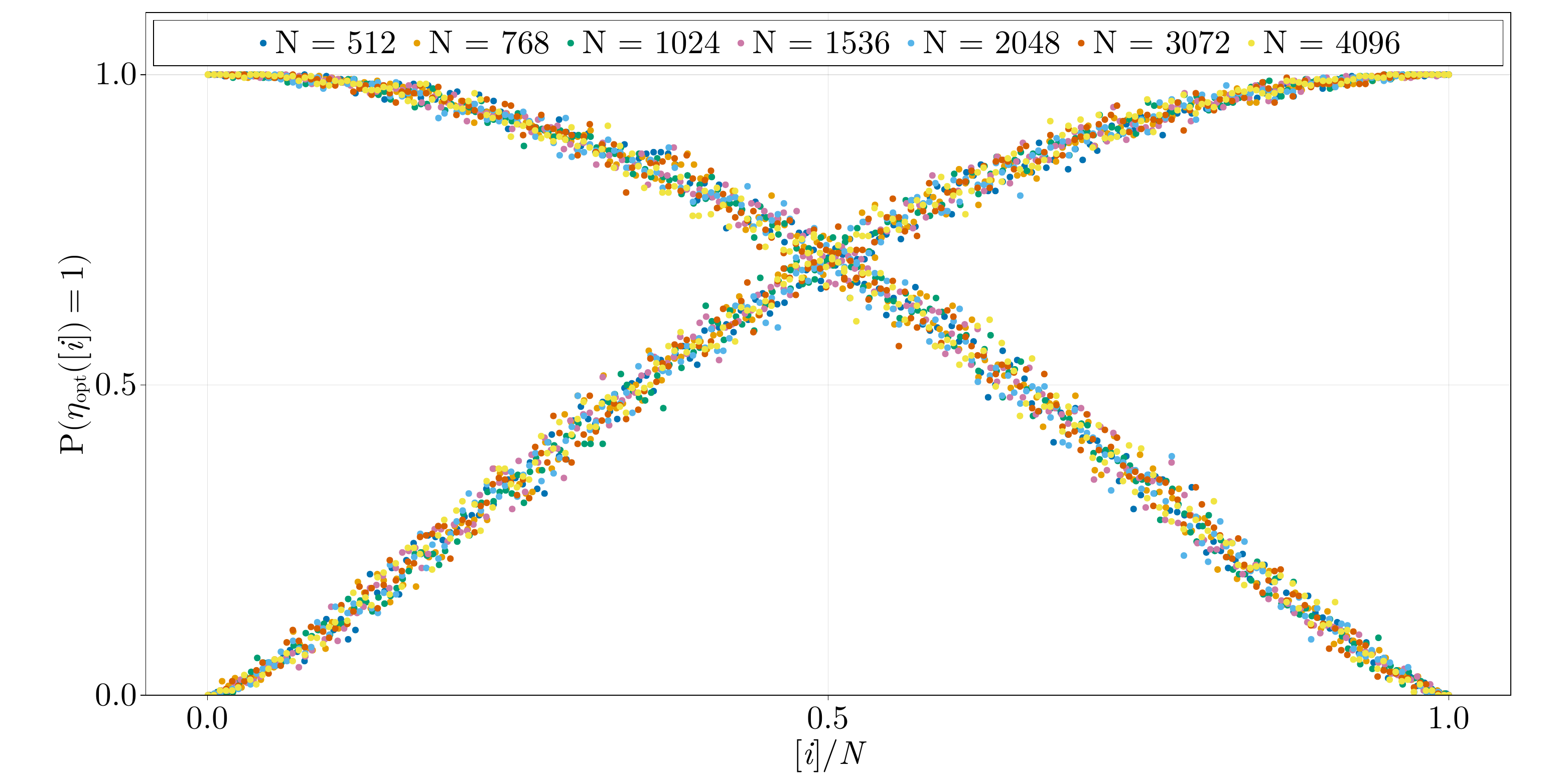}
        \caption{Probability the $[i]$-th row of $\tilde{J}$ belongs to the minimizer
                (the increasing cloud of points) and to the maximizer 
                (the decreasing cloud of points) of $H$ for
                several values of $N$. 
                Estimates are obtained as averages 
                over 400 realization of $J$
                for each value of $N$.
                The shape of the cloud appears to
                be independent of the size of the system.}
        \label{fig:ith_row_in_etamax}
\end{figure}

Finally, we observe that the random variables 
$\mathbbm{1}_{\{\eta^{\min}_{[i]} = 1\}}$
and $\mathbbm{1}_{\{\eta^{\max}_{[i]} = 1\}}$ appear to be positively correlated
for all $i$. The qualitative behavior of the strength of the correlation
between these two variables as a function of $[i]$ is provided by the
estimates of Fig.~\ref{fig:i_in_min_and_max_correlation}
\begin{figure}[H]
    \centering
            \includegraphics[width=0.78\textwidth]{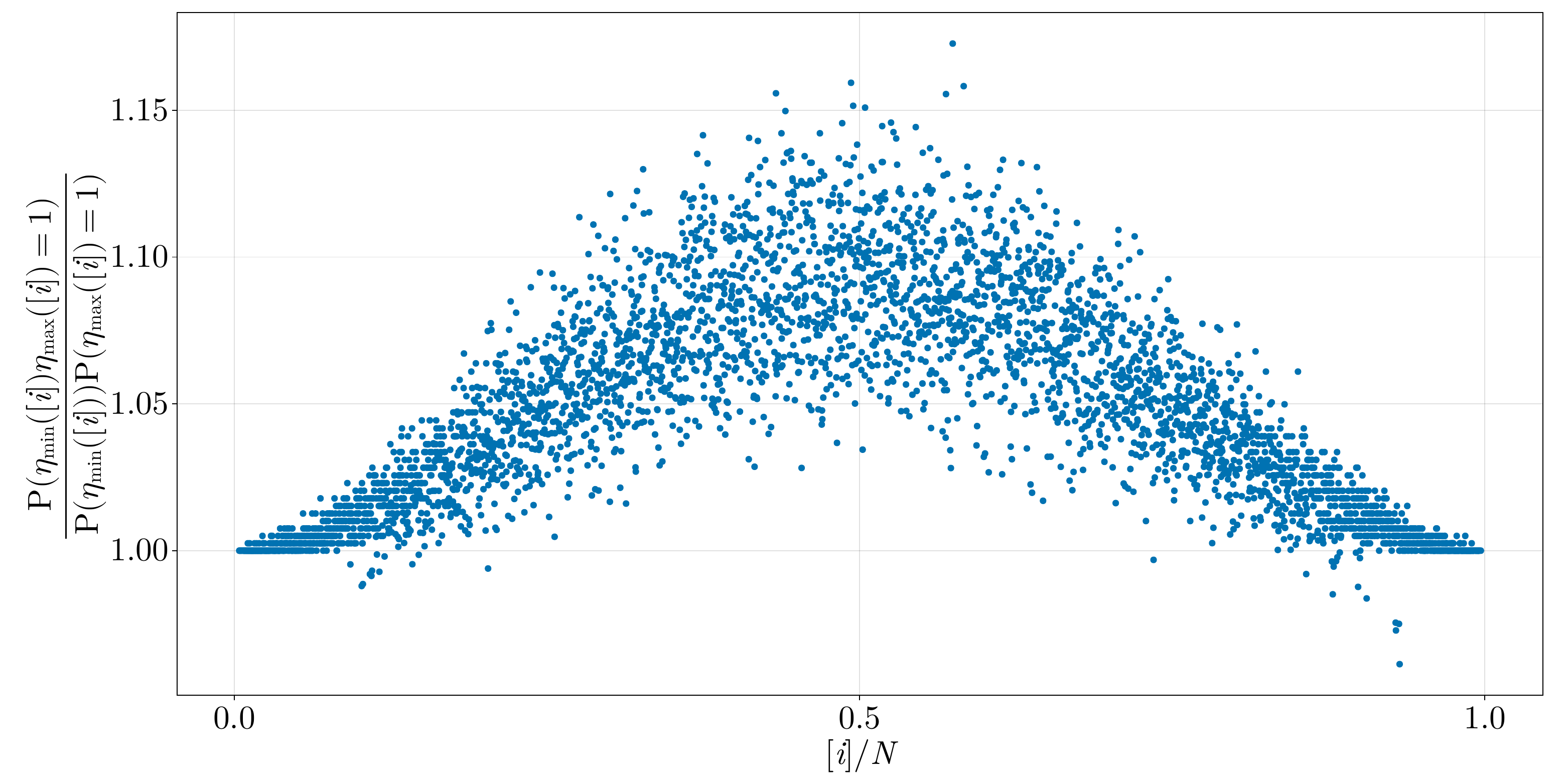}
        \caption{Ratio between the probability that the $[i]$-th belongs to both
                the minimizer and the maximizer of $H$ and the product of the
                probability of each of the two events for $N = 4096$.
                Estimates are obtained as averages 
                over 400 realizations of $J$
                for each value of $N$.}
        \label{fig:i_in_min_and_max_correlation}
\end{figure}

\section{Proofs}\label{sec:proofs}
We state first some general results that will be used
below in the proofs of 
Theorem~\ref{thm:min_max_self_averaging} and 
Theorem~\ref{thm:universality_min_max}.

Let $X = \{X_{1}, \ldots, X_{n}\}$ be independent random variables
and let $W=g\left(X_1, \ldots, X_i, \ldots, X_n\right)$
with $g$ a measurable function.
Further, 
let $X^{\prime} = \{X_{1}^{\prime}, \ldots, X_{n}^{\prime}\}$ be an independent copy of 
$X_{1}, \ldots, X_{n}$ and write
\begin{equation}
    W_{i}=g\left(X_{1}, \ldots, X_{i}^{\prime}, \ldots, X_{n}\right)
\end{equation}

Define the random variable $V$ by
\begin{equation}
    V=\mathbb{E}\left[\sum_{i=1}^n\left(W - W_{i}\right)^{2} \mid {X}\right] 
\end{equation}
which allows to re-state Efron-Stein's theorem (see \cite{efron1981jackknife})
as follows
\begin{theorem}[Efron-Stein] \label{thm:efronstein}
    \begin{equation}
        \variance*{W} \leq \frac{1}{2} \expect*{V}
    \end{equation}
\end{theorem}

From \cite{boucheron2003concentration}, we can bound the moment
generating function with the following
\begin{theorem}\label{thm:mgf}
For all $\theta>0$ and $\lambda \in(0,\frac{1}{\theta})$,
    \begin{equation}
        \expect*{\exp (\lambda(W - \expect*{W}))} 
            \leq \exp\left(\frac{\lambda \theta}{1-\lambda \theta}\right)
                \expect*{\exp \left(\frac{\lambda V}{\theta}\right)}.
    \end{equation}
\end{theorem}

Moreover, a straightforward consequence of \cite[Theorem~2]{boucheron2005moment}
yields
\begin{theorem}\label{thm:bound_third_moment_w} 
    \begin{equation}
        \expect*{\abs*{W-\expect*{W}}^3} < \expect*{\abs*{V}^{\frac{3}{2}}}.
    \end{equation}
\end{theorem}

To prove both Theorem~\ref{thm:min_max_self_averaging} and Theorem~\ref{thm:universality_min_max} a key role is played by
the \emph{free energy} function 
$F : \mathbb{R}^{N \times N} \rightarrow \mathbb{R}$
defined as
\begin{equation}\label{eq:free_energy}
    F_{\beta}(X) := \beta^{-1} \log Z(X).
\end{equation}
with $Z(X) := \sum_{\eta} e^{\beta H(X, \eta)}$, 
$H(X, \eta) = \frac{1}{\sqrt{N}}\sum_{i,j} X_{i,j} \eta_{i} \eta_{j}$.
and $X = \{X_{i,j}\}_{1 \leq i, j, \leq N}$ a collection
of independent random variables (the matrix of the couplings).
We do not write the dependence of $F_{\beta}$ and $Z$ on $N$ to make
the notation less heavy.
Further, we set
\begin{equation}
    p(X, \eta) := \frac{e^{\beta H(X, \eta)}}{Z_{X}};
    \quad    
    \average*{\eta_i\eta_j}{X} = \sum_{\eta}\eta_{i}\eta_{j}p(X, \eta)
\end{equation}
where $p(X, \eta)$ is the \emph{Gibbs measure}
of $\eta$ and
$\average*{}{X}$ denotes the expectation with respect to the Gibbs
measure when the matrix of couplings is $X$.

In the next lemma, we determine bounds on the derivatives of $F$ with
respect to one of the couplings.
\begin{lemma}\label{thm:bounds_derivatives_F}
    Let $F : \mathbb{R}^{N \times N} \rightarrow \mathbb{R}$ be as in
    \eqref{eq:free_energy}.
    Then
    \begin{equation} 
        \abs*{\frac{\partial F_{\beta}(X)}{\partial X_{i,j}}}\leq \frac{1}{\sqrt{N}}\, ;
        \quad
        \abs*{\frac{\partial^{2} F_{\beta}(X)}{\partial X_{i,j}^{2}}}\leq \frac{\beta}{4N}\, ;
        \quad
        \abs*{\frac{\partial^{3} F_{\beta}(X)}{\partial X_{i,j}^{3}}}
             \leq \frac{\beta^{2}}{6 \sqrt{3} N^{3/2}}
    \end{equation}
\end{lemma}

\begin{proof}
    At first, observe that a straightforward computation gives
    \begin{equation}\label{eq:preliminary_derivatives}
    \frac{\partial H(X, \eta)}{\partial X_{i,j}} = \frac{1}{\sqrt{N}} \eta_{i} \eta_{j} \, ;
    \quad   
    \frac{\partial Z}{\partial X_{i,j}} (X) 
        = \frac{\beta}{\sqrt{N}}\sum_{\eta} \eta_{i} \eta_{j}  e^{\beta H(X, \eta)}
        =\frac{\beta}{\sqrt{N}}\average*{\eta_i \eta_j}{X} Z(X).
    \end{equation}
    Using \eqref{eq:preliminary_derivatives}, the bound on the first 
    derivative of $F$ with respect to $X_{i,j}$ is readily obtained from the following
    \begin{equation}
        \frac{\partial F_{\beta}(X)}{\partial X_{i,j}}
            = \frac{1}{\beta Z(X)}\frac{\partial Z(X)}{\partial X_{i,j}}
            =\frac{1}{\sqrt{N}} \average*{\eta_{i}\eta_{j}}{X}
    \end{equation}
    by observing that $\average{\eta_{i}\eta_{j}}{X} \in [0, 1]$.
    To compute higher order derivatives note that
    \begin{equation}
        \frac{\partial p(X, \eta)}{\partial X_{i,j}}
        =
        \frac{ \frac{\beta}{\sqrt{N}}\eta_{i}\eta_{j} e^{\beta H(X, \eta)} Z- 
        \frac{\beta}{\sqrt{N}}\average*{\eta_{i}\eta_{j}}{X} e^{\beta H(X, \eta)} Z}{ Z^2}
        =
        \frac{\beta}{\sqrt{N}}\left(\eta_{i}\eta_{j}-\average*{\eta_{i}\eta_{j}}{X}\right) p(X, \eta)
    \end{equation}
    which, in turn, yields,
    \begin{align}
        \frac{\partial \average*{\eta_{i}\eta_{j}}{X}}{\partial X_{i,j}}
        & = \sum_{\eta} \eta_{i}\eta_{j} \frac{\partial p}{\partial X_{i,j}} (X, \eta)
          = \frac{\beta}{\sqrt{N}}
            \sum_{\eta} \eta_{i}\eta_{j} \left(\eta_{i}\eta_{j}
                                               -\average*{\eta_{i}\eta_{j} }{X}\right) p(X, \eta) \\
        & = \frac{\beta}{\sqrt{N}}\left(\average*{\eta_{i}\eta_{j} }{X}- \average*{\eta_{i}\eta_{j} }{X}^2\right) 
          = \frac{\beta}{\sqrt{N}}\average*{\eta_{i}\eta_{j}}{X}\left(1- \average*{\eta_{i}\eta_{j} }{X}\right)
    \end{align}
    and
    \begin{align}
        \frac{\partial^2 \average*{\eta_{i}\eta_{j}}{X}}{\partial X_{ij}^2}
        = \frac{\beta}{\sqrt{N}} 
          \frac{\partial\left(\average*{\eta_{i}\eta_{j}}{X}
                              \left(1 - \average*{\eta_{i}\eta_{j}}{X}\right)
                        \right)}
               {\partial X_{ij}}
        = \frac{\beta^2}{N} 
           \average*{\eta_{i}\eta_{j}}{X}\left(1 - \average*{\eta_{i}\eta_{j}}{X}\right)
           \left(1 - 2\average*{\eta_{i}\eta_{j}}{X}\right).
    \end{align}
    Consequently,
    \begin{gather}
        \frac{\partial^{2} F_{\beta}(X)}{\partial X_{i,j}^{2}} 
        = \frac{\beta}{N}\average*{\eta_{i}\eta_{j}}{X}
          \left(1 - \average*{\eta_{i}\eta_{j} }{X}\right)
          \quad \text{ and } \quad\\
        \frac{\partial^{3} F_{\beta}(X)}{\partial X^{3}}
        = \frac{\beta^2}{N^{3/2}}\average*{\eta_{i}\eta_{j}}{X}
                                 (1 - \average*{\eta_{i}\eta_{j}}{X})
                                 \left(1 - 2\average*{\eta_{i}\eta_{j}}{X}\right).
    \end{gather}
    The claim follows by observing that, in the interval $[0, 1]$,  
    $x (1 - x) \leq \frac{1}{4}$ and
    $\abs{x (1 - x) (1 - 2x)} \leq \frac{1}{6\sqrt{3}}$
\end{proof}

\subsection{Proof of Theorem~\ref{thm:min_max_self_averaging}}
\label{sec:proof_thm_min_max_self_averaging}
We give the proof for the convergence of $m_{\max, N}$.
The proof for $m_{\min, N}$ is analogous.

In the rest of this section, we write
$\overbar{H}(J) := \max_{\eta} H(J, \eta)$

To prove the convergence in probability of $m_{\max, N}$ to its expectation we use the following results which allow us to bound the variance of the free energy.

Assume $J_{i,j}$ are independent identically distributed 
random variables with 
expected value zero and variance $1$.
Call $F = F_{\beta}(J) = \frac{1}{\beta} \log \sum_{\eta}
            e^{\beta\sum_{i,j} J_{i,j} \eta_{i} \eta_{j}}$.
The following theorem provides a bound on the variance of $F$.
\begin{theorem}\label{thm:var_F}
    $\variance*{F} \leq N$
\end{theorem}
\begin{proof}
    Let $F_{i,j}$ be obtained from $F$ by substituting $J_{i,j}$ with $J'_{i,j}$ (an independent copy of $J_{i,j}$). Then, by Lemma~\ref{thm:bounds_derivatives_F}, 
    $\abs*{F - F_{i,j}} \leq \frac{1}{\sqrt{N}} \abs*{J_{i,j} - J'_{i,j}}$.
    From Theorem~\ref{thm:efronstein} it follows
    \begin{align}
        \variance*{F} 
            & \leq \frac{1}{2} \expect*{\sum_{i,j}(F - F_{i,j})^{2}}
                         = \frac{1}{2} \sum_{i,j} \expect*{(F - F_{i,j})^{2}}
                         = \frac{1}{2} N^{2} \expect*{(F - F_{1,1})^{2}}\\
            & \leq \frac{1}{2} N^{2} 
              \expect*{ \left( \frac{1}{\sqrt{N}} \abs*{J_{i,j} - J'_{i,j}} \right)^{2} }
              = N \expect*{J_{1,1}^{2}} = N
    \end{align}
    Note that this estimate is uniform in $\beta$.
\end{proof}
Since $\overbar{H}(J) = \lim_{\beta \to \infty} F_{\beta}(J)$, as
an immediate consequence we get
\begin{corollary}\label{thm:variance_max_per_particle}
    $\variance*{\overbar{H}(J)} \leq N$.
\end{corollary}

By Chebyshev's inequality and the previous corollary we get
\begin{align}
    \proba*{ \abs*{ \frac{\overbar{H}(J)}{N} - \expect*{\frac{\overbar{H}(J)}{N} }} > \varepsilon  }
    \leq \frac{\variance*{ \frac{\overbar{H}(J)}{N} } }{\varepsilon^{2}}
    = \frac{1}{N^{2}} \frac{\variance*{\overbar{H}(J)}}{\varepsilon}
    \leq \frac{1}{N \varepsilon}
\end{align}
from which claim \ref{thm:min_max_self_averaging_conv_prob} follows.

To prove the other two claims,
set $V_{F} = \expect*{\sum_{i,j} (F - F_{i,j})^{2} \mid J}$.
We have
\begin{align}\label{eq:bound_V_F}
    V_{F}
          & = \sum_{i,j} \expect*{(F - F_{i,j})^{2} \mid J} 
            \leq \sum_{i,j} \frac{1}{N} \expect*{(J_{i,j} - J'_{i,j})^{2} \mid J_{i,j}} \\
          & = \frac{1}{N} \sum_{i,j} \left( J_{i,j}^{2} + \expect*{(J'_{i,j})^2} \right) 
            = \frac{1}{N} \sum_{i,j} \left( J_{i,j}^{2} + 1 \right)
\end{align}

Let us examine the case $\expect*{\abs*{J_{1,1}}^{3}} < \infty$.

Leveraging on Jensen's inequality we have
$\abs*{\sum_1^n a_i}^{\frac{3}{2}} \leq \sqrt{n}\sum_{1}^{n} \abs*{a_i}^{\frac{3}{2}}$.
The latter inequality can be used to show that
\begin{align}
    \expect*{V_{F}^{\frac{3}{2}}} 
    & \leq \expect[\Big]{\abs[\Big]{\frac{1}{N}\sum_{i,j} \left(J_{i,j}^2 + 1\right)}^{\frac{3}{2}}} 
      \leq  \frac{1}{N^{\frac32}} \expect[\Big]{ N \sum_{i,j} \abs*{J_{i,j}^2 + 1}^{\frac{3}{2}}} \\
    & \leq \frac{1}{\sqrt{N}} \expect*{\sum_{i,j} \sqrt{2} (\abs*{J_{i,j}}^3 + 1)}
      =\sqrt{2}N^{\frac{3}{2}} \expect*{\abs*{J_{1,1}}^3 + 1}.
\end{align}
Then, by Theorem~\ref{thm:bound_third_moment_w} and for some constant $C$,
\begin{equation}
    \expect*{\abs[\bigf]{F - \expect{F}}^3} \leq CN^{\frac{3}{2}}
\end{equation}
which, in turn, implies
\begin{equation}\label{eq:almost_sure_convergence_F}
    \expect*{ \abs*{ \frac{F-\expect{F}}{N} }^{3} } \leq CN^{-\frac{3}{2}}
\end{equation}
Since \eqref{eq:almost_sure_convergence_F} does not depend on $\beta$,
it holds for $\overbar{H}(J) = \lim_{\beta \to \infty} F_{\beta}(J)$
as well. 
Chebyshev's inequality and Borel-Cantelli lemma allow us to conclude the proof of claim \ref{thm:min_max_self_averaging_conv_as}.

Consider now the case of strictly subgaussian $J_{i,j}$'s.
Recall 
that for a centered strictly subgaussian random variable $X$
with variance $\sigma^{2}$,
\begin{equation}\label{eq:bound_mgf_subgaussian}
    \expect*{\exp(sX)} \leq \exp\left( \frac{\sigma^{2}s^{2}}{2}\right),
    \quad \forall s \in \mathbb{R}.
\end{equation}
Also, $\forall r \in \mathbb{N}$, $\expect*{X^{2r}} \le 2^{r+1}\sigma^{2r}r!$,
yielding
\begin{align}
\expect*{e^{s\left(X^2-\sigma^2\right)}} 
    & =1+s \expect*{X^2-\sigma^2} 
       +\sum_{r=2}^{\infty} \frac{s^r \expect*{\left(X^2-\sigma^2\right)^r}}{r !} 
     \leq 1+\sum_{r=2}^{\infty} \frac{s^r \expect*{X^{2 r}}}{r !} \label{eq:bound_using_central_moments}\\
    & \leq 1+2\sum_{r=2}^{\infty} s^r 2^r \sigma^{2 r} 
     =1+\frac{8 s^2 \sigma^4}{1-2 s \sigma^2}
\end{align}
where the inequality in \eqref{eq:bound_using_central_moments}
follows from the fact that, for a non negative random variable (such as $X^{2}$),
the central $r$-th moment is smaller or equal to the $r$-th moment (recall that $X^{2}$ has expected value $\sigma^{2}$).

This result, together with \eqref{eq:bound_V_F},
can be used to bound the moment generating function
of $V_{F}$ as
\begin{align}\label{eq:bound_mgf_V_F}
    \expect*{\exp \left(t V_{F}\right)} 
        &\leq  \expect*{\exp \left(\frac{t}{N}\sum_{i,\,j} \left(X_{ij}^2 - 1 + 2\right) \right)}
         = \left( e^{\frac{2t}{N}} \expect*{\exp \left(\frac{t}{N}\left(X_{1\,1}^2 - 1\right) \right) }\right)^{N^2}\\ 
        & \leq e^{\frac{2t}{N}} \left[1+\frac{1}{N}\frac{8 t^2} {N- 2t}\right]^{N^2}\leq e^{9t^2}
\end{align}
for $N$ large enough
where the last inequality follows from the fact that
$\left[1+\frac{1}{N}\frac{8 t^2} {N- 2t}\right]^{N^2}$ is a decreasing sequence
with the same limit as $\left[1 + \frac{8 t^2 }{N^2}\right]^{N^2}$.
Then, from Theorem~\ref{thm:mgf} we get, for all $\theta>0$ and $\lambda \in (0, \frac{1}{\theta})$,
\begin{equation}
\expect*{\exp (\lambda(F-\expect*{F}))} \leq 
\exp\left(\frac{\lambda\theta}{1-\lambda\theta}\right)  \expect*{\exp \left(\frac{\lambda V_{F}}{\theta}\right)} 
\leq \exp\left(\frac{\lambda\theta}{1-\lambda\theta}+\frac{9\lambda^2}{\theta^2}\right)
\end{equation}
and hence, by exponential Markov inequality
\begin{equation}
\proba*{\abs[\big]{F-\mathbb{E}[F]} > t }
\leq 2\exp\left(-\lambda t + \frac{\lambda\theta}{1-\lambda\theta}+\frac{9\lambda^2}{\theta^2}\right)
\leq e^{-at}
\end{equation}
for some $a>0$ (by optimizing on $\lambda, \theta$) and for every $t > 0$.
Setting $t = Nz$, with $z > 0$ we get
\begin{equation}\label{eq:exponential_convergence_F}
    \proba*{\abs[\big]{F-\mathbb{E}[F]} > Nz }
    \leq e^{-aNz}.
\end{equation}

Since also
\eqref{eq:exponential_convergence_F} does not depend on $\beta$,
it holds for $\overbar{H}(J) = \lim_{\beta \to \infty} F_{\beta}(J)$.
Dividing by $N$ we get claim \ref{thm:min_max_self_averaging_conv_expfast}. \qed

\subsection{Proof of Theorem~\ref{thm:universality_min_max}}
\label{sec:proof_thm_universality_min_max}

Using the bounds of Lemma~\ref{thm:bounds_derivatives_F} we can prove
that the expectation of the free energy does not depend 
on the distribution of the couplings 
in the thermodynamic limit.

As for Theorem~\ref{thm:universality_min_max},
let $J = \left\{ J_{i,j} \right\}_{1\leq i,j\leq N}$ 
and  $Y = \left\{ Y_{i,j} \right\}_{1\leq i,j\leq N}$ 
be two independent sequences of 
independent random variables,  
such that for every $i, j$ 
$\expect*{J_{i,j}} = \expect*{Y_{i,j}} = 0$ and 
$\expect*{J_{i,j}^{2}} = \expect*{Y_{i,j}^{2}} = 1$.

\begin{theorem}\label{thm:universality_free_energy}
    Consider $F_{\beta}(J)$ as in \eqref{eq:free_energy}.
    Then
    \begin{equation}
    \lim_{N \to \infty} \left| \expect*{F_{\beta}(J)} - \expect*{F_{\beta}(Y)} \right| = 0 
    \end{equation}
\end{theorem}

\begin{proof}
    The proof is a straightforward adaptation of 
    Chatterjee's extension of Lindeberg's argument 
    for the central limit theorem. 
    See \cite{chatterjee2005simple, chatterjee2006generalization} 
    for a comprehensive treatment.

    Let $0 \leq k \leq n=N^{2}$ be any numbering of the elements of the sequences
    and define 
    \begin{gather}
    \mathbf{Z}_k:=\left(J_{1}, \ldots, J_{k-1}, J_k, Y_{k+1}, \ldots, Y_{n}\right) \\
    \mathbf{W}_k:=\left(J_{1}, \ldots, J_{k-1}, 0, Y_{k+1}, \ldots, Y_{n}\right)
    \end{gather}
    Now consider a Taylor expansion of $F_{\beta}$ around
    $J_{k} = 0$ and write
    \begin{gather}
    F_{\beta}(\mathbf{Z}_k)
        = F_{\beta}(\mathbf{W}_k)
            + J_{k} \frac{\partial F_{\beta}(\mathbf{W}_k)}{\partial J_{k}}
            + \frac{1}{2} J_{k}^2 \frac{\partial^{2} F_{\beta}(\mathbf{W}_k)}{\partial J_{k}^{2}} 
            + R_{k}\\ 
    F_{\beta}(\mathbf{Z}_{k-1}))
        = F_{\beta}(\mathbf{W}_k)
            + Y_{k} \frac{\partial F_{\beta}(\mathbf{W}_k)}{\partial J_{k}}
            + \frac{1}{2} Y_{k}^{2} \frac{\partial^{2} F_{\beta}(\mathbf{W}_k)}{\partial J_{k}^{2}}
            + S_{k}.
    \end{gather}

The bounds from Lemma ~\ref{thm:bounds_derivatives_F} then give
\begin{gather}
\left| R_k\right| \leq \frac{\beta}{8N} X_k^2, \quad
\left| S_k\right|  \leq \frac{\beta}{8N} Y_k^2 \qquad \text{and} \\
\left| R_k\right| \leq  \frac{\beta^{2}}{36 \sqrt{3} N^{3/2}} |X_k|^3,  \quad
\left| S_k\right| \leq  \frac{\beta^{2}}{36 \sqrt{3} N^{3/2}} |Y_k|^3.
\end{gather}

Noticing that, for each $k$, $J_{k}, Y_{k}$ and $\mathbf{W}_k$ are independent
and recalling that $J_{k}$ and $Y_{k}$ have both mean zero and variance $1$,
we get
\begin{equation}
\expect*{F_{\beta}(\mathbf{Z}_k)- F_{\beta}(\mathbf{Z}_{k-1}))}
    = \expect*{R_{k}- S_{k}}
\end{equation}
which gives the estimate
\begin{dmath}\label{eq:maggiorazione_secondo}
\abs*{\expect*{F_{\beta}(X)} - \expect*{F_{\beta}{(Y)}}}
	= \abs*{\sum_{i=1}^n \expect*{F_{\beta}(\mathbf{Z}_k) - F_{\beta}(\mathbf{Z}_{k-1})}}
	\leq \sum_{i=1}^n \expect*{\abs*{R_k}} + \expect*{\abs*{S_k}}\\
		+ {
			\frac{\beta}{8N}\sum_{i=1}^n \left[\mathbb{E}\left(X_i^2, \left|X_i\right|>K\right) 
											+ \mathbb{E}\left(Y_i^2 ;\left|Y_i\right|>K\right)\right]
		  }
		+ {\frac{\beta^{2}}{36 \sqrt{3} N^{3/2}}\sum_{i=1}^n \left[\mathbb{E}\left(\left|X_i\right|^3 ;\left|X_i\right| \le K\right)
			+\mathbb{E}\left(\left|Y_i\right|^3 ;\left| Y_i \right| \le K\right)\right]}
\end{dmath}

If we furthermore assume $\gamma < \infty$, we only need to use third order Taylor bounds on $R_{k}, S_{k}$
getting the more explicit bound

    \begin{equation}\label{eq:maggiorazione_terzo}
    \left| \expect*{F_{\beta}(J)} - \expect*{F_{\beta}(Y)} \right|
            \leq \frac{\beta^2 \gamma}{18\sqrt{3}}{\sqrt{N}} 
    \end{equation}

\end{proof}

Finally, we can prove Theorem~\ref{thm:universality_min_max}.
We write the proof for the maximum of $H$. The proof for
the minimum can be done in the same way by
replacing $\beta$ with $-\beta$.

Recall that we set $\overbar{H}(J) = \max_{\eta} H(J, \eta)$.

We have
\begin{align} 
\overbar{H}(J) & =\beta^{-1} \log \left[e^{\beta\overbar{H}(J)}\right]   
    \leq \beta^{-1} \log \left[\sum_{\eta} e^{\beta H(J, \eta)}\right]   
    \leq \beta^{-1} \log \left[2^N e^{\beta \overbar{H}(J)}\right]
\end{align}
getting the uniform bound
\begin{equation}
    \abs*{\overbar{H}(J) - F_{\beta}(J)} \leq \beta^{-1}  N \log 2.
\end{equation}

From Theorem~\ref{thm:universality_free_energy}
\begin{equation}\label{ground_state}
\begin{split}
    \abs*{\expect*{\overbar{H}(J)} - \expect*{\overbar{H}(Y)}}
    & \leq \abs*{\expect*{\overbar{H}(J) - F_{\beta}(J)}} 
           + \abs*{\expect*{F_{\beta}(J)} - \expect*{F_{\beta}(Y)}}
           + \abs*{\expect{\overbar{H}(Y)} - F_{\beta}(Y)} \\
    & \leq \frac{2 N \log 2}{\beta} + \abs*{\expect*{F_{\beta}(J)} - \expect*{F_{\beta}(Y)}}
\end{split}
\end{equation}

First assume $\gamma < \infty$. From \eqref{eq:maggiorazione_terzo} and \ref{ground_state}
\begin{align}
    \abs*{\expect*{\overbar{H}(J)} - \expect*{\overbar{H}(Y)}}
    & \leq \frac{2 N \log 2}{\beta} + \frac{\beta^2\gamma}{18\sqrt{3}} \sqrt{N}
\end{align}

By choosing $\beta=N^{1/6}$ we get the second part of the thesis.

We now consider $\gamma=\infty$. Define $g(K):=\max_{i, j}\left( \mathbb{E}\left(J_{i j}^2 ;\left|J_{i j}\right|>K\right), 
\mathbb{E}\Big(Y_{i j}^2 ;\left|Y_{i j}\right|>K\right) \Big)$.
From the finite variance hypothesis, $g(K) \to 0$ as $K \to \infty$. In  \eqref{eq:maggiorazione_secondo} we take 
$\beta = g(N^{\frac{1}{4}})^{-\frac{1}{2}}$,
so that $\beta \to \infty$ as $N \to \infty$. 
Also observe that  $\gamma=\infty$ implies $Kg(K) \to \infty$ as $K\to \infty$.

Estimate \eqref{eq:maggiorazione_secondo}, with $K=N^{\frac{1}{4}}$, now becomes
\begin{equation}
\frac1N \abs*{\expect*{F_{\beta}(X)} - \expect*{F_{\beta}{(Y)}}} \leq
C_1 g(N^{\frac{1}{4}})^{\frac{1}{2}}+
C_2 \left(N^{\frac{1}{4}}g(N^{\frac{1}{4}})\right)^{-1}
\end{equation}

Which vanishes as $N \to \infty$ thus proving also part 1 of Theorem \ref{thm:universality_min_max}.  \qed

\section{Conclusions and open problems}\label{sec:conclusions}
The analysis carried over in this paper unveils several lines of investigation
that we believe are rather interesting.

It would be useful to identify the minimal 
requirement on the distribution of the couplings to have
the minimum and maximum per particle to converge to their
expected values exponentially fast.

One feature we find particularly intriguing is the relationship
between the elements of $J$ contributing to the minimum and those
contributing to the maximum of $H$ 
discussed in Conjecture~\ref{thm:blocks_features}.
We think that a rigorous understanding of the relative
size of the blocks of $J$ and the correlation between the rows belonging
to the minimizer and the maximizer of $J$ could significantly improve
our understanding of the problem.

Better understanding would, in turn, provide useful tools in order to evaluate
the heuristic algorithms used to tackle QUBO.

Finally, it would be interesting to study the problem when
the $J_{i,j}$ are not independent, especially in the case
where the \emph{interaction graph} is not the
complete graph.

\section*{Acknowledgement}
We thank two anonymous referees, whose comments contributed to significantly improve the content of the paper.
The authors are grateful to the Department of Mathematics and Physics of the University of Rome "Tre" which provided us with additional computing resources.
AT thanks PRIN 2022 PNRR: "RETINA: REmote sensing daTa INversion with multivariate functional modeling for essential climAte variables characterization" (Project Number: P20229SH29, CUP: J53D23015950001) funded by the European Union under the Italian National Recovery and Resilience Plan (NRRP) of NextGenerationEU, under the Italian Ministry of University and Research (MUR). 
BS acknowledges the support of MUR Excellence Department Project
MatMod@TOV awarded to the Department of Mathematics, University of Rome 
Tor Vergata, CUP E83C23000330006

\section*{Data Availability}
Data of Table~\ref{tab:m_N_uniform_benchmarks} is available from
\cite{wang2013probabilistic} and references therein. 
All statistics concerning the generated data during
this study are included in this article.

\section*{Compliance with Ethical Standards}
The authors declare no conflict of interest.

\end{document}